\documentclass[3p,times,11pt,preprint]{elsarticle}
\RequirePackage{amsmath,amssymb,amsthm}
\pdfoutput=1
\usepackage{booktabs}
\usepackage[svgnames]{xcolor}
\usepackage{url}

\biboptions{sort&compress}

\newtheorem{thm}{Theorem}[section]
\newtheorem{cor}[thm]{Corollary}
\newtheorem{lemma}[thm]{Lemma}
\newtheorem{hyp}{Hypothesis}
\theoremstyle{definition}
\newtheorem{defn}[thm]{Definition}
\theoremstyle{remark}
\newtheorem{rem}[thm]{Remark}
\newtheorem{example}[thm]{Example}
\numberwithin{equation}{section}
\renewcommand{\(}{\langle\mskip2mu\relax}
\renewcommand{\)}{\mskip2mu \rangle}
\newcommand{\pminus}{\phantom{-}}
\newcommand{\zwidth}[2][0pt]{\hbox to #1{\hss#2\hss}}
\DeclareMathOperator{\GL}{GL}
\DeclareMathOperator{\SL}{SL}
\DeclareMathOperator{\Sym}{Sym}
\DeclareMathOperator{\Fix}{Fix}

\newcommand{\Q}{\mathbb{Q}}
\newcommand{\C}{\mathbb{C}}
\newcommand{\Z}{\mathbb{Z}}
\newcommand{\rou}{\boldsymbol\mu}
\newcommand{\Magma}{\textsf{Magma}}
\let\sdprod=\rtimes

\nopreprintlinetrue

\begin{document}
\begin{frontmatter}
\title{Normalisers of parabolic subgroups in finite unitary reflection groups}

\author{Muraleedaran Krishnasamy}
\ead{kmuraleedaran@sydney.edu.au}
\author{D.~E.~Taylor}
\ead{donald.taylor@sydney.edu.au}

\address{School of Mathematics and Statistics\\
The University of Sydney\\
NSW 2006, Australia}

\begin{abstract}
It is well known that the normaliser of a parabolic subgroup of a finite
Coxeter group is the semidirect product of the parabolic subgroup by the
stabiliser of a set of simple roots.  We show that a similar result holds
for all finite unitary reflection groups: namely, the normaliser of a
parabolic subgroup is a semidirect product and in many, but not all, cases
the complement can be obtained as the stabiliser of a set of roots. In
addition we record when the complement acts as a reflection group on the
space of fixed points of the parabolic subgroup.

\begin{keyword}%
unitary reflection group\sep complex reflection group\sep parabolic
subgroup\sep normaliser
\par\leavevmode\emph{2010 Mathematics subject classification:} 20F55
\end{keyword}
\end{abstract}

\end{frontmatter}

\section{Introduction}\label{sec:intro}
Let $V$ be a finite-dimensional vector space over the complex numbers $\C$
and suppose that $V$ is equipped with a positive definite hermitian form
$(-,-)$. A \emph{reflection} is a linear transformation of $V$ of finite
order whose fixed point space is a hyperplane. A \emph{unitary reflection
group} on $V$ is a group of linear transformations of $V$ generated by
reflections which preserve $(-,-)$.

In general a reflection of order $m$ has the form
\[
   r_{a,\zeta}(v) = v - (1-\zeta)\frac{(v,a)}{(a,a)} a,
\]
where $a$ is a non-zero element of $V$ called a \emph{root} of $r_{a,\zeta}$
and $\zeta\in\C$ is a primitive $m$th root of unity.

A subgroup of $G$ is a \emph{reflection subgroup} if it is generated by its
reflections. A \emph{parabolic subgroup} of $G$ is the pointwise stabiliser
$G(X)$ of a subset of $X$ of $V$; by a theorem of Steinberg
\cite{steinberg:1964} $G(X)$ is a reflection subgroup. Given a parabolic
subgroup $P = G(X)$, let $U = \Fix_V(P)$ be the space of fixed points of $P$.
Then $G(U) \subseteq G(X) = P\subseteq G(U)$ whence $P = G(U)$. A
straightforward calculation shows that the normaliser $N_G(P)$ of $P$ in $G$
is $G_U$, the setwise stabiliser of $U$. Similarly, $N_G(P) = G_{U^\perp}$,
where $U^\perp$ is the orthogonal complement of $U$ in~$V$.

\vfill
\pagebreak
Given a parabolic subgroup $P$ of $G$, it is the purpose of this
paper---extending the work in \cite{muralee:2005}---to describe the structure
of the normaliser $N_G(P)$. In particular, we show that in all cases there
exists a complement $H$ to $P$ in its normaliser; that is, $N_G(P)$
is a semidirect product $P\sdprod H$. For brevity we refer to $H$ as a 
\emph{parabolic complement}.

If $G$ is imprimitive, a parabolic complement is constructed using the methods
of \cite{taylor:2012}. If $G$ is primitive, there is almost always a set of
roots of reflections which generate $P$ and whose stabiliser is a complement;
the exceptions are certain parabolic subgroups of the Shephard and Todd
groups $G_{29}$ and $G_{31}$ and the parabolic subgroups of rank 1 excluded
by Lemma \ref{lemma:stab}.

In most cases the existence of a parabolic complement is obtained without resort
to computer-assisted calculations.  However, for 17 of the 212 conjugacy classes
of parabolic subgroups of the primitive reflection groups a specific but
somewhat ad hoc choice of roots is given in Table~\ref{tbl:eis}.  The
choice of roots and the verification that their setwise stabiliser is a
parabolic complement was carried out using \Magma\ \cite{magma:1997} code.  

The \Magma\ code is part of a general
package\footnote{\url{http://www.maths.usyd.edu.au/u/don/details.html#programs}}
which was written to facilitate further exploration and provide some of the
information in the tables of \S\ref{sec:tab}.  The root systems and 
generating reflections used by \Magma\ satisfy the defining relations
corresponding to the diagrams in 
\cite{broue-malle-rouquier:1998}.  For unitary reflection groups other 
choices of generators have been proposed 
\cite{basak:2012,bessis-bonnafe-rouquier:2002,malle-michel:2010} but 
except for the parabolic complements described in Table~\ref{tbl:eis}
our results do not depend on the particular choice of generators.

Throughout the paper $H^\circ$ denotes the subgroup of $H$ generated by the
elements that act as reflections on $\Fix_V(P)$ and $Q$ is the parabolic
subgroup $G(U^\perp)$.

The complexification of a real reflection group is a Coxeter group and in
this case the structure of the normaliser of a parabolic subgroup is well
known \cite{brink-howlett:1999,howlett:1980}. Standard facts about unitary
reflection groups can be found in \cite{kane:2001} and
\cite{lehrer-taylor:2009}.

\section{Reduction to the irreducible case}\label{sec:reduction}
Suppose that $G$ is a finite unitary reflection group on the vector space $V$
and that $G$ preserves the positive definite hermitian form $(-,-)$. From
\cite[Theorem~1.27]{lehrer-taylor:2009} $V$ is the direct sum of pairwise
orthogonal $G$-invariant subspaces $V_1$, $V_2$, \dots,~$V_k$ such that the
restriction of $G$ to $V_i$ is an irreducible reflection group  $G_i$ on
$V_i$ and $G\simeq G_1\times G_2\times\dots \times G_k$.

\begin{thm}
Suppose that $P$ is a parabolic subgroup of $G$,
where $G = G_1\times G_2\times\dots \times G_k$, as above. For $1\le i\le k$,
let $U_i = \Fix_{V_i}(P)$ and $P_i = G_i(U_i)$. Then
\begin{enumerate}[(i)]
\item $\Fix_V(P) = U_1\oplus U_2\oplus \dots\oplus U_k$;
\item $P = P_1\times P_2\times \dots \times P_k$;
\item $P_i$ is the parabolic subgroup of $G_i$ generated by the reflections
    in $P$ whose roots belong to $V_i$;
\item $N_G(P) = N_{G_1}(P_1)\times N_{G_2}(P_2)\times\dots \times
    N_{G_k}(P_k)$;
\item if $H_i$ is a complement to $P_i$ in $N_G(P_i)$ for $1\le i\le k$,
then $H_1\times H_2\times\dots\times H_k$ is a complement to $P$ in~$N_G(P)$.
\end{enumerate}
\end{thm}

\begin{proof}
(i) Suppose that $g\in G$ and $v = v_1 +\dots + v_k$, where $v_i\in V_i$.
Since $V_i$ is $G$-invariant we have $g(v) = v$ if and only if $g(v_i) = v_i$
for all $i$. This proves (i) and since $g = g_1g_2\cdots g_k$ with $g_i\in
G_i$, the same calculation shows that $g\in P$ if and only if $g_i\in P_i$
for all $i$, thus proving (ii). Then (iii) follows from \cite[Theorem
1.27]{lehrer-taylor:2009}.

Finally, (iv) follows from the fact that $g(U) = U$ if and only if $g_i(U_i)
= U_i$ for all~$i$ and then (v) follows from (iv).
\end{proof}

As a consequence of this theorem we may restrict our attention to parabolic
subgroups of irreducible unitary reflection groups.

\smallskip
According to Shephard and Todd \cite{shephard-todd:1954} (see
\cite{cohen:1976,cohen:1978,lehrer-taylor:2009} for more recent proofs), the
finite irreducible unitary reflection groups are:
\begin{enumerate}[\ (i)]
\item the symmetric groups, $\Sym(n)$,
\item imprimitive groups $G(m,p,n)$ for $m,n \ge 2$, except $G(2,2,2)$,
\item cyclic groups,
\item 34 primitive groups, of rank at most 8, labelled $G_4$, $G_5$, \dots,
    $G_{37}$.
\end{enumerate}

A cyclic reflection group has no proper parabolic subgroups and we treat the
remaining cases separately, beginning with the symmetric and imprimitive
groups, culminating in Theorem \ref{thm:main-imprim}.  The main results on
parabolic subgroups of the primitive reflection groups can be found in
Theorems \ref{thm:rk1} and \ref{thm:prim3}. Further information about the
parabolic complements is summarised in the tables of \S\ref{sec:tab}.

\section{The imprimitive unitary reflection groups $G(m,p,n)$}%
\label{sec:imprimitive}

\begin{defn}\label{defn:imprim}
A group $G$ acting on a vector space $V$ is \emph{imprimitive} if, for some
$k > 1$, $V$ is a direct sum of non-zero subspaces $V_i$ ($1\le i \le k$)
such that the action of $G$ on $V$ permutes $V_1$, $V_2$, \dots,~$V_k$ among
themselves; otherwise $G$ is \emph{primitive}.  The set $\Omega =
\{V_1,V_2,\dots,V_k\}$ is a \emph{system of imprimitivity} for~$G$. If $G$
acts transitively on $\Omega$, then $\Omega$ is a \emph{transitive} system of
imprimitivity.
\end{defn}

The imprimitive unitary reflection group $G(m,p,n)$ may be defined as
follows. Let $V$ be a complex vector space of dimension $n$ with a positive
definite hermitian form $(-,-)$.  Let $I = \{1,2,\dots,n\}$, let $\Lambda =
\{\,e_i \mid i\in I\,\}$ be an orthonormal basis for $V$, and let $\rou_m$ be
the group of $m$th roots of unity.  Given a function $\theta : I \to \rou_m$,
the linear transformation that maps $e_i$ to $\theta(i) e_i$ ($1\le i\le n$)
will be denoted by $\hat\theta$.

If $p$ divides $m$, let $A(m,p,n)$ be the group of all linear transformations
$\hat\theta$ such that $\prod_{i=1}^n \theta(i)^{m/p} = 1$. For $\sigma\in
\Sym(n)$, the action of $\sigma$ on $V$ is defined by $\sigma(e_i) =
e_{\sigma(i)}$.

The group $G(m,p,n)$ is the semidirect product of $A(m,p,n)$ by the symmetric
group $\Sym(n)$ with the action on $V$ given above.  When the elements of
$G(m,p,n)$ are written as pairs $(\sigma;\theta)$, the multiplication becomes
\[
  (\sigma_1;\theta_1)(\sigma_2;\theta_2) = (\sigma_1\sigma_2;
    \theta_1^{\sigma_2}\theta_2),
\]
where for $\sigma\in\Sym(n)$ and $\theta : I \to \rou_m$ we define
$\theta^{\sigma}(i) = \theta(\sigma(i))$.

The group $G(m,1,n)$ is the wreath product $\rou_m\wr\,\Sym(n)$ and
$G(m,p,n)$ is a normal subgroup of index $p$ in $G(m,1,n)$. If $n > 1$,
$G(m,p,n)$ is imprimitive and $\{V_1,V_2,\dots,V_n\}$ is a transitive
system of imprimitivity where $V_i = \C e_i$. The group $G(m,p,1) =
G(m/p,1,1)$ is cyclic of order $m/p$. The centre of $G(m,p,n)$ consists of
scalar matrices and its order is $m\gcd(p,n)/p$.

The group $G(1,1,n)$ is isomorphic to $\Sym(n)$ and its elements are the
pairs $(\sigma;\textbf 1)$, where $\sigma\in\Sym(n)$ and $\textbf 1(i) = 1$
for all~$i$. In particular, if $\sigma_i$ is the transposition $(i,i+1)$, the
element $r_i = (\sigma_i;\textbf1)$ is a reflection of order 2, which
interchanges $e_i$ and $e_{i+1}$ and fixes $e_j$ for $j\ne i, i+1$.

Fix a primitive $m$th root of unity $\zeta$ and let $t = (1;\theta_\zeta)$,
where $\theta_\zeta(1) = \zeta$ and $\theta_\zeta(i) = 1$ if $i\ne 1$. Then
$t$ is a reflection of order $m$ and $s = t^{-1}r_1t$ is a reflection of
order 2, which interchanges $e_1$ and $\zeta e_2$.  These reflections
generate the groups $G(m,p,n)$ (see \cite[Chap 2, \S7]{lehrer-taylor:2009}):
\begin{equation}\label{eqn:gens}
\begin{aligned}
 G(1,1,n) &= \( r_1, r_2, \dots, r_{n-1}\),\\
 G(m,m,n) &= \( s, r_1, r_2, \dots, r_{n-1}\),\\
 G(m,1,n) &= \( t, r_1, r_2, \dots, r_{n-1}\),\\
 G(m,p,n) &= \( s, t^p, r_1, r_2, \dots, r_{n-1}\)\quad\text{for $p\ne1,m$}.
\end{aligned}
\end{equation}

The \emph{support} of a reflection subgroup is the subspace spanned by the
roots of its reflections. The group $G(1,1,n)$ is imprimitive in its action
on $V$ but for $n \ge 5$ the action on its support is primitive.

\begin{defn}
We write $\lambda \vdash n$ to mean that the sequence $\lambda =
(n_1,n_2,\dots,n_d)$ is a partition of~$n$; that is, $n_1$, $n_2$,
\dots,~$n_d$ are positive integers such that $n_1\ge n_2\ge\cdots \ge n_d$
and $n = n_1+n_2 +\cdots + n_d$. We also use the notation $1^{b_1}\,2^{b_2}
\cdots n^{b_n} \vdash n$ to denote a partition of $n$ with $b_i$ parts of
size $i$ for $1\le i\le n$.
\end{defn}

\begin{defn}\label{defn:part}
Suppose that $\Xi = (\Lambda_0,\Pi,\xi)$, where $\Lambda_0$ is a subset of
$\Lambda$, $\Pi = (\Lambda_1, \Lambda_2, \dots,\Lambda_d)$ is a partition
of $\Lambda\setminus\Lambda_0$ and $\xi : \Lambda\setminus\Lambda_0 \to
\rou_m$. Let $n_i = |\Lambda_i|$ and define the subgroup $P_\Xi$ of
$G(m,p,n)$ by
\[
  P_\Xi = P_0\times P_1\times \cdots\times P_d,
\]
where $P_0$ is the reflection group $G(m,p,n_0)$ acting on the subspace
spanned by $\Lambda_0$ and for $1\le i\le d$, $P_i$ is the reflection group
$G(1,1,n_i)\simeq \Sym(n_i)$ permuting the vectors $\{\,\xi(e)e \mid e \in
\Lambda_i\,\}$. The factor $P_0$ is omitted if $\Lambda_0 = \emptyset$. If $m
= 1$, we require $\Lambda_0 = \emptyset$.

By construction $P_\Xi$ is the pointwise stabiliser in $G(m,p,n)$ of the
subspace spanned by the vectors $\sum_{e\in \Lambda_i}\xi(e) e$ for $1\le
i\le d$ and therefore $P_\Xi$ is a parabolic subgroup.
\end{defn}

\begin{defn}
If $n_0 = |\Lambda_0|$ and if $\lambda$ is the partition of $n - n_0$
obtained by writing the quantities $|\Lambda_i|$ (for $1\le i\le d$) in
descending order, the \emph{type} of $P_\Xi$ is $(n_0,\lambda)$.
\end{defn}

\begin{lemma}\label{lemma:para}
Given $n_0 \le n$ and $\lambda = 1^{b_1}\,2^{b_2}\cdots (n-n_0)^{b_{n-n_0}} \vdash
n - n_0$, the number of parabolic subgroups of $G(m,p,n)$ of type
$(n_0,\lambda)$ is
\begin{equation}\label{eqn:orbit}
  \frac{n!\,m^{n-n_0-d}}{n_0!\;b_\lambda}
\end{equation}
where $d = b_1 + b_2 + \dots + b_{n-n_0}$ and where
\[
  b_\lambda = b_1!\,b_2!\,\cdots\,b_{n-n_0}!\;(1!)^{b_1}(2!)^{b_2}\cdots
  ((n-n_0)!)^{b_{n-n_0}}.
\]
\end{lemma}

\begin{proof}\relax(cf.\ \cite[\S 6]{orlik-terao:1992})
We count the triples $(\Lambda_0,\Pi,\xi)$ of type $(n_0,\lambda)$. Triples
$(\Lambda_0,\Pi,\xi)$ and $(\Lambda_0,\Pi,\xi')$ define the same parabolic
subgroup if and only if, for $1 \le i\le d$, the vectors $\sum_{e\in
\Lambda_i}\xi(e) e$ and $\sum_{e\in \Lambda_i}\xi'(e) e$ are equal up to a
non-zero scalar multiple.  Therefore, the number of parabolic subgroups is
\[
  \binom{n}{n_0}\frac{(n-n_0)!}{b_\lambda} m^{n-n_0-d}
  = \frac{n!\,m^{n-n_0-d}}{n_0!\;b_\lambda}.\qedhere
\]
\end{proof}

\begin{thm}
A subgroup $P$ of $G(m,p,n)$ is parabolic if and only if $P = P_\Xi$ for some
$\Xi = (\Lambda_0,\Pi,\xi)$.
\end{thm}

\begin{proof}
This follows from \cite[Lemma 3.3 and Theorem 3.11]{taylor:2012}.
\end{proof}

\begin{defn}\label{defn:std}
Given $n_0$ such that $0\le n_0 < n$ and $\lambda =
(n_1,n_2,\dots,n_d)\vdash n - n_0$ let $k_{-1} = 0$ and for $0\le i \le d$
let $k_i = n_0 + n_1 +\cdots +n_i$. For $1\le i\le d$ define $\Lambda_i = \{
e_j \mid k_{i-1} < j \le k_i\}$ and let $\Pi = (\Lambda_1,\Lambda_2,
\dots,\Lambda_d)$.

The \emph{standard parabolic subgroup} of type $(n_0,\lambda)$ is
$P_{(n_0,\lambda)} = P_{(\Lambda_0,\Pi,\textbf1)}$, where $\textbf1(e) = 1$
for all $e\in\Lambda\setminus\Lambda_0$. For consistency with Definition
\ref{defn:part} we require $n_0 = 0$ if $m = 1$.
\end{defn}

As a reflection group
\begin{equation}\label{eqn:standard}
  P_{(n_0,\lambda)} =  G(m,p,n_0) \times \prod_{i=1}^{d} G(1,1,n_i)
\end{equation}
where $G(m,p,n_0)$ acts on the subspace spanned by $\Lambda_0$ and where
$G(1,1,n_i)$ acts on the subspace of $V$ spanned by $\Lambda_i$. If $n_i =
1$, the factor $G(1,1,n_i)$ is trivial and if $n_0 = 0$ there is no factor
$G(m,p,n_0)$.

\begin{defn}
Given $\alpha\in\rou_m$, define $\theta_\alpha : I \to \rou_m$ by
$\theta_\alpha(1) = \alpha$ and $\theta_\alpha(i) = 1$ for $i > 1$. Then
$\hat\theta_\alpha\in G(m,1,n)$ and we define
\begin{equation}\label{eqn:A}
  P_{(n_0,\lambda)}^\alpha =
  \hat\theta_\alpha P_{(n_0,\lambda)}\hat\theta_\alpha^{-1}.
\end{equation}
\end{defn}

\begin{thm}[{\cite[Theorem 2.1]{taylor:2012}}]
A subgroup of $G(m,p,n)$ is parabolic if and only if it is conjugate to
$P_{(n_0,\lambda)}^\alpha$ for some $n_0$, $\lambda = (n_1,n_2,\dots,n_d)$
and $\alpha$. Furthermore, if $n_0\ne 0$, then $P_{(n_0,\lambda)}^\alpha$ is
conjugate in $G(m,p,n)$ to $P_{(n_0,\lambda)}$ for all $\alpha$, whereas the
groups $P_{(0,\lambda)}^\alpha$ and $P_{(0,\lambda)}^\beta$ are conjugate in
$G(m,p,n)$ if and only if $\alpha\beta^{-1}\in \rou_{m/e}$ where $e =
\gcd(p,n_1,\dots,n_d)$.
\end{thm}

\begin{cor}\label{cor:orders}
If $G = G(m,p,n)$ and $P = P_{(n_0,\lambda)}$, where $\lambda =
(n_1,n_2,\dots,n_d)\vdash n - n_0$ has $b_i$ parts of size $i$ for $1\le i\le
n-n_0$, then
\[
  |N_G(P)/P\,| = \begin{cases}
    \prod_{i=1}^{n-n_0} m^{b_i}\,b_i!&n_0 \ne 0\\[9pt]
    \dfrac{e}{p}\prod_{i=1}^n m^{b_i}\,b_i!&n_0 = 0,
  \end{cases}
\]
where $e = \gcd(p,n_1,\dots,n_d)$.
\end{cor}

\begin{proof}
We have $|G(m,p,n)| = m^n n!/p$ and from \eqref{eqn:standard} we have
\[
  |P\,| = \begin{cases}
    \dfrac 1p m^{n_0} n_0!\prod_{i=1}^d n_i!&n_0 \ne 0\\[9pt]
    \prod_{i=1}^d n_i!&n_0 = 0.
  \end{cases}
\]
Thus the corollary follows from equation \eqref{eqn:orbit} combined with the
previous theorem.
\end{proof}

Observe that if $n_0 \ne 0$ or $p$ divides $n_i$ for all $i\ge 0$,
$|N_G(P)/P\,| = \prod_{i,b_i\ne 0}|G(m,1,b_i)|$. In this case we shall
see that the order coincidence can be lifted to an isomorphism of groups.

\begin{defn}
Suppose that $\lambda = (n_1,n_2,\dots,n_d) = 1^{b_1}\,2^{b_2}\cdots
n^{b_n}\vdash n-n_0$. The group $G(1,1,n_i)$ in the decomposition
\eqref{eqn:standard} of the standard parabolic subgroup $P =
P_{(n_0,\lambda)}$ of $G(m,p,n)$ acts on the subspace spanned by $\Lambda_i$
and fixes the sum $\Sigma\Lambda_i$ of the vectors in~$\Lambda_i$.

Given $k$ such that $b_k\ne 0$, let $M_k$ be the subspace spanned by the
vectors in the $b_k$ blocks of size $k$ of the partition $\Pi$ defining $P$.
Then $P$ acts on $M_k$ as the reflection group $P_k = G(1,1,k)^{b_k}$ and $P$
is the direct product of the $P_k$. We call the groups $P_k$ the
\emph{isotypic components} of $P$.

The vectors $\Sigma\Lambda_i$, where $\Lambda_i$ is a block of size $k$, form
a basis for $X_k = \Fix_{M_k}(P)$. The subspaces $M_k$ and $X_k$ are
invariant under the action of the normaliser of $P$ in $G(m,p,n)$.
\end{defn}

Let $\Gamma_1^k$, $\Gamma_2^k$, \dots, $\Gamma_{b_k}^k$ be the blocks of size
$k$ in the partition of $I$ corresponding to $\Pi$, labelled so that the
$j$th block is
\[
  \Gamma_j^k = \{i_{1,j}^k, i_{2,j}^k, \dots, i_{k,j}^k\}.
\]

The group $\Sym(b_k)\simeq G(1,1,b_k)$ is generated by transpositions and the
reflection $\sigma\in G(1,1,b_k)$ corresponding to the transposition $(r,s)$
induces
\begin{equation}\label{eqn:vtransp}
  v^{(k)}(r,s) = (\tau;\textbf1)\text{ in $G(1,1,n)$},
\end{equation}
where $\tau = (i_{1,r}^k,i_{1,s}^k)(i_{2,r}^k,i_{2,s}^k) \cdots
(i_{k,r}^k,i_{k,s}^k)$. Then $v^{(k)}(r,s)$ belongs to the normaliser of $P$
in $G(m,p,n)$ and acts on $X_k$ as the reflection $\sigma$.

Recall that in the paragraph preceding \eqref{eqn:gens} we chose a fixed
primitive $m$th root of unity $\zeta\in\rou_m$. Define $\theta_j^{(k)} :
\Gamma_j^k \to \rou_m$ by $\theta_j^{(k)}(g) = \zeta$ if $g\in \Gamma_j^k$
and $\theta_j^{(k)}(g) = 1$ otherwise.  Then
\begin{equation}\label{eqn:vcore}
  h^{(k)}(j) = (1;\theta_j^{(k)}) \in A(m,1,n)
\end{equation}
normalises $P$ and acts on $X_k$ as $(1;\theta_j)$, where $\theta_j :
\{1,2,\dots,b_k\} \to \rou_m$ sends $j$ to $\zeta$ and all other elements
to~1.

\smallskip
In effect we have $M_k = U\otimes W$, where $U$ is a vector space of
dimension $k$ and $W$ is a vector space of dimension $b_k$. The action of
$G(1,1,k)$ on $U$ extends naturally to an action of $G(1,1,k)^{b_k}$ on
$U\otimes W$, which is normalised by the action of $G(m,1,b_k)$ on $W$.

Recall that when $H$ is a complement to $P$ in $N_G(P)$, $H^\circ$ denotes
the subgroup generated by the elements of $H$ that act as reflections on
$\Fix_V(P)$.  The identification of $H^\circ$ as a reflection group in the
following theorem was previously obtained by Amend et 
al.~\cite[Proposition 4.7]{amend-etal:2016}.

\begin{thm}\label{thm:main-imprim}
Suppose that $G = G(m,p,n)$ and let $P = P_{(n_0,\lambda)}$.  Let $V$ be the
vector space on which $G$ acts and let $X_k$ be the fixed point space of the
$k$th isotypic component $P_k$ of $P$ acting on the space $M_k$ spanned by
$\Lambda_k$ as in Definition~\ref{defn:std}. If $b_k\ne 0$, let $\varphi_k$
be the restriction homomorphism $N_G(P) \to \GL(X_k)$. Then
\begin{enumerate}[(i)]
\item $N_G(P)$ has a complement $H$ to $P$.
\item If $n_0\ne 0$, then $\varphi_k(H) = G(m,1,b_k)$ and $H = H^\circ$
    acts on $\Fix_V(P) = \bigoplus_k X_k$ as the reflection group
    $\prod_{k,b_k\ne 0} G(m,1,b_k)$.
\item If $n_0 = 0$, then $\varphi_k(H^\circ) = G(m,p_k,b_k)$, where $p_k =
    p/\gcd(p,k)$ and $H^\circ$ acts faithfully on $\Fix_V(P)$ as $\prod_{k,
    b_k\ne 0}G(m,p_k,b_k)$. Furthermore, for all $k$ such that $b_k\ne 0$,
    $\varphi_k(H) = G(m,q_k,b_k)$ for some divisor $q_k$ of~$p_k$.
\end{enumerate}
\end{thm}

\begin{proof}
We begin by considering the case $p = 1$.

The transformations $v^{(k)}(r,s)$ and $h^{(k)}(j)$ defined in
\eqref{eqn:vtransp} and \eqref{eqn:vcore} are image of the generators of
$G(m,1,b_k)$ and thus for all $g\in G(m,1,b_k)$ there is an element $\hat
g\in G(m,1,n)$ that acts on $X_k$ as $g$ and fixes $M_\ell$ pointwise for all
$\ell\ne k$.  Thus $g\mapsto \hat g$ is a monomorphism from $G(m,1,b_k)$ into
the normaliser $N$ of $P = P_{(n_0,\lambda)}$ in $G(m,1,n)$ and
$\varphi_k(\hat g) = g$.

On comparing orders, using Corollary \ref{cor:orders}, it follows that the
subgroup
\[
  H = \{\,\hat g \mid g \in  \prod_{k,b_k\ne 0} G(m,1,b_k)\,\}
\]
is a complement to $P$ in $N$. Furthermore, $H$ acts on $\Fix_V(P)$ as the
reflection group $\prod_{k,b_k \ne 0} G(m,1,b_k)$.  This proves (i) and (ii)
for the case $p = 1$ and so from now on we may suppose that $p\ne 1$.

Suppose that $n_0\ne 0$. The reflections $v^{(k)}(r,s)$ of
\eqref{eqn:vtransp} belong to $G(1,1,n) \subset G(m,p,n)$.  Since $n_0 \ne 0$
we have $e_1\in\Lambda_0$ and therefore $t^{-k}h^{(k)}(j)\in
G(m,m,n)\subseteq G(m,p,n)$, where $t$ is the generator of $G(m,1,n)$ which
is used in \eqref{eqn:gens}. The elements $v^{(k)}(r,s)$ and
$t^{-k}h^{(k)}(j)$ belong to the normaliser of $P$ in $G(m,p,n)$ and we
define $H$ to be the subgroup generated by $v^{(k)}(r,s)$ and
$t^{-k}h^{(k)}(j)$ for all $k$ such that $b_k\ne 0$. Then $H$ is isomorphic
to $\prod_{k,b_k \ne 0} G(m,1,b_k)$ and $\varphi_k(H) = G(m,1,b_k)$. A
comparison of orders shows that $H$ is a parabolic complement.

For the remainder of the proof we may suppose that $p\ne 1$ and $n_0 = 0$.
Let $\tilde G = G(m,1,n)$ and let $X = \Fix_V(P)$. Because $n_0 = 0$, $\tilde
G(X)$ is the direct product of the same groups $G(1,1,k)$ defining $P$ and
therefore $\tilde G(X) = P$.  As shown above, the normaliser of $P$ in
$\tilde G$ has a complement $\tilde H$. Then $H = G\cap \tilde H$ is a
parabolic complement, completing the proof of (i).

We shall use the notation established in the first part of this proof. If $g
= (1;\theta)\in A(m,1,b_k)$ where $\theta : \{1,2,\dots,b_k\} \to \rou_m$,
then $\hat g\in A(m,1,n)$ belongs to $A(m,p,n)$ if and only if
\[
  \prod_{j=1}^{b_k} \theta(j)^{km/p} = 1.
\]
In this case $\phi_k(\hat g)\in A(m,p_k,b_k)$ where $p_k = m/\gcd(m,km/p) =
p/\gcd(p,k)$.

We have $\varphi_k(H^\circ) = G(m,p_k,b_k)$ and it follows that $H^\circ$
acts faithfully on $\Fix_V(P)$ as $\prod_{k=1}^n G(m,p_k,b_k)$. In addition,
$\tilde H = \prod_{k,b_k\ne 0} G(m,1,b_k)$ and $\tilde A = \prod_{k,b_k\ne 0}
A(m,1,b_k)$ is a normal subgroup of $\tilde H$. Thus $A = G\cap \tilde A$ is
a normal subgroup of $H$ and $\varphi_k(A)$ contains $A(m,m,b_k)$. It follows
that $\varphi_k(H) = G(m,q_k,b_k)$ for some divisor $q_k$ of $p_k$.
\end{proof}

\begin{example}
Suppose that $G = G(2,2,n)$; that is, $G$ is a Coxeter group of type $D_n$.
Let $\lambda = 1^{b_1}\,2^{b_2}\cdots n^{b_n} \vdash n$. The complement $H$
to $P = P_{(0,\lambda)}$ in its normaliser is generated by reflections in its
action on $\Fix_V(P)$ if and only if there is at most one value of $k$ such
that $k$ is odd and $b_k \ne 0$.
\end{example}

\begin{example}
Suppose that $p \ne 1$, $n_0 = 0$ and that $P$ has a single isotypic
component $G(1,1,k)^{b_k}$. Then $\lambda = k^{b_k}\vdash kb_k = n$ and the
subgroup $H^\circ$ of the parabolic complement is
isomorphic to $G(m,p/s,b)$, where $s = \gcd(p,k)$. It follows from Corollary
\ref{cor:orders} that $H^\circ$ is a parabolic complement for~$P$.
\end{example}

\begin{example}
Next suppose that $p \ne 1$, $n_0 = 0$ and $\lambda = 1^{b_1}\,k^{b_k}\vdash
b_1 + kb_k$. Once again $P$ is the direct product of $b_k$ copies of
$G(1,1,k)$. But now the subgroup $H^\circ$ of the parabolic complement $H$
is $H^\circ = H_1^\circ\times H_2^\circ$, where
$H_1^\circ\simeq G(m,p,b_1)$, $H_2^\circ\simeq G(m,p/s,b_k)$, where $s =
\gcd(p,k)$ and the index of $H^\circ$ in $H$ is~$p/s$.

Put $x = v^1(1)^s h^{(k)}(1)^\ell$ where $s = pr - k\ell$ for some $r$
and~$\ell$. The product of the diagonal entries of $x$ is $\zeta^{s+k\ell}$
and therefore $x\in G(m,p,n)$. By construction $x$ normalises $P$,
$H_1^\circ$ and $H_2^\circ$ and on comparing orders we see that $H =
\(H^\circ,x\)$.

The space of fixed points of $P$ has a decomposition $\Fix_V(P) = X_1\oplus
X_2$ corresponding to the isotypic components of $P$. The image of the
restriction of $H_1 = \(H_1^\circ,x\)$ to $X_1$ is the reflection group
$G(m,s,b_1)$.  Recall that $t = (1;\theta)$ is the reflection defined prior
to \eqref{eqn:gens} and that its determinant is $\zeta$.  We have $t^{p/s}\in
H_2^\circ$ and therefore the restriction of $H_2 = \(H_2^\circ,x\)$ to $X_2$
is the reflection group $G(m,1,b_k)$ because $(\zeta^\ell)^{k/s} =
\zeta^{(p/s)r}\zeta$.
\end{example}

\begin{example}
Suppose that $p \ne 1$, $n_0 = 0$ and $\lambda = k_1^{b_1}\,k_2^{b_2}\vdash
n$, where $n = k_1b_1 + k_2b_k$.  Then $P = P_1\times P_2$, where $P_i$ is
the direct product of $b_i$ copies of $G(1,1,k_i)$. The subgroup $H^\circ$ of
the parabolic complement $H$ is $H^\circ = H_1^\circ\times
H_2^\circ$, where $H_i^\circ\simeq G(m,p_i,b_i)$ and where $p_i =
p/\gcd(p,k_i)$. The index of $H^\circ$ in $H$ is
$p\gcd(p,k_1,k_2)/\gcd(p,k_1)\gcd(p,k_2)$.
\end{example}

\section{Root systems and parabolic subgroups of Coxeter type}
Let $P$ be a proper parabolic subgroup of the reflection group $G$ and let $X
= \Fix_V(P)$. Then $N_G(P) = G_X = G_Y$, where $Y = X^\perp$ is the
orthogonal complement of $X$ in $V$.  If $r\in N_G(P)$ is a reflection and
$a$ is a root of $r$, then $a\in X\cup Y$ and therefore the reflection
subgroup of $N_G(P)$ is $P\times Q$, where $Q$ is the parabolic subgroup
$G(Y)$. This symmetry between $P$ and $Q$ is echoed in the tables in
\S\ref{sec:tab}.

When $G$ is a real reflection group (namely a finite Coxeter group) a
parabolic complement to $P$ can be obtained as the setwise stabiliser of a
set of (simple) roots corresponding to reflections that generate $P$ (see
\cite[Cor.~3]{howlett:1980}). In this section we show that in many cases
there is a generalisation of this result to parabolic subgroups of unitary
reflection groups. Consequently, the complement contains the subgroup~$Q$.

However, there are unitary reflection groups which contain parabolic
subgroups $P$ such that $Q$ cannot be extended to a parabolic complement. 
Therefore, in these cases, even though it turns out that the
normaliser of $P$ is a semidirect product, a complement cannot be realised as
the stabiliser of a set of roots of $P$.

\medskip
Let $F$ be a finite abelian extension of $\Q$, let $A$ be the ring of integers
of $F$, let $\rou$ be the group of roots of unity in $A$ and let $U$ be a vector
space over $F$ with an hermitian product $(-,-)$ such that $V = \C\otimes_F U$.
An $A$-root system for the reflection group $G$, as defined in
\cite[Definition 1.43]{lehrer-taylor:2009}, is a pair $(\Sigma,f)$ such that
$\Sigma$ spans $U$, $0\notin \Sigma$, and $f : \Sigma \to \rou$. Furthermore,
$G$ is generated by the reflections $\{\,r_{a,f(a)} \mid a \in\Sigma\}$, and
\begin{enumerate}[(i)]
\item for all $a\in\Sigma$ and all $\lambda\in F$ we have $\lambda a\in\Sigma$
    if and only if $\lambda\in\rou$;
\item for all $a\in\Sigma$ and all $\lambda\in \rou$ we have $f(\lambda a) = 
    f(a) \ne 1$;
\item for all $a,b\in\Sigma$ we have $(1-f(b))(a,b)/(b,b)\in A$;
\item for all $a,b\in\Sigma$ we have $r_{a,f(a)}(b)\in\Sigma$ and 
    $f(r_{a,f(a)}(b)) = f(b)$.
\end{enumerate}

Let $\zeta_m$ denote a primitive $m$th root of unity and put $\rou_m = \(\zeta_m\)$.
A $\Z[\zeta_m]$-root system for $G(m,m,n)$ is $\Sigma(m,m,n) = \{\,\xi e_i - \eta e_j
\mid\text{$\xi,\eta\in\rou_m$ and $i\ne j$}\,\}$, where $f(a) = -1$ for all roots.
To obtain a root system $\Sigma(m,p,n)$ for $G(m,p,n)$ where $p \ne m$ include
the roots $\pm\xi e_i$ for $1\le i\le n$ and $\xi\in\rou_m$ and define $f(\pm\xi
e_i) = \zeta_m^p$. If $f(a) = -1$, let $r_a$ denote the reflection with root $a$.
See \cite{lehrer-taylor:2009} for descriptions of root systems for the primitive
reflection groups.

\begin{defn}
A parabolic subgroup $P$ of $G$ is of \emph{Coxeter type} $X$ if it is
generated by a set $S$ of reflections such that $(P,S)$ is a Coxeter system
and $X$ is the standard name of $(P,S)$. The standard names are described
in~\S\ref{sec:tab}.
\end{defn}

Suppose that $P$ is a parabolic subgroup of Coxeter type in the reflection
group $G$ and that $\Sigma$ is an $A$-root system for $G$.  Let $\Sigma_P$ be
the set of roots in $\Sigma$ corresponding to the reflections in $P$.   Let
$\Phi$ denote the root system of $P$ as in \cite{humphreys:1990}.  At this
point we work with the following hypotheses.

\begin{hyp}
The root system $\Phi$ is a subset of $\Sigma_P$.
\end{hyp}

It is clear that it is possible to choose a root system $\Phi$ for the
standard parabolic subgroup $P_{(0,\lambda)}$ of $G(m,p,n)$ so that Hypothesis~1
is satisfied.  Later we shall show that it holds for many other groups.

\begin{hyp} The group of roots of unity of $A$ is $\rou = \(-1\)\times
\(\gamma\)$ where the order $h$ of $\gamma$ is odd.
\end{hyp}

Let $\Delta$ be a set of simple roots in $\Phi$ and let $\Phi^+$ be the
corresponding set of positive roots. Define
\begin{equation}\label{eqn:extend}
  \widetilde\Delta = \bigcup_{i=0}^{h-1}\gamma^i\Delta,
  \quad\widetilde\Phi^+ = \bigcup_{i=0}^{h-1}\gamma^i\Phi^+
  \quad\text{and}\quad\widetilde\Phi^- = -\widetilde\Phi^+.
\end{equation}

\begin{lemma}\label{lemma:length}
For $g\in N_G(P)$, define $N(g) = \{\,a\in\Phi^+\mid g(a)\in
\widetilde\Phi^-\,\}$. For $a\in\Delta$, if $g(a)\in\widetilde\Phi^-$, then
$r_a N(gr_a) = N(g)\setminus \{a\}$, whereas if $g(a)\in\widetilde\Phi^+$,
then $N(gr_a) = r_aN(g)\mathbin{\dot\cup}\{a\}$.
\end{lemma}

\begin{proof}
This is a direct consequence of the well-known fact \cite[Proposition
1.4]{humphreys:1990} that for Coxeter groups
$r_a(\Phi^+\setminus\{a\}) = \Phi^+\setminus\{a\}$.
\end{proof}

The proof of the following theorem is based on \cite[Lemma 2 and Corollary
3]{howlett:1980}.

\begin{thm}\label{thm:paracomp}
Suppose that $P$ is a parabolic subgroup of Coxeter type in the reflection
group $G$ and that $G$ has an $A$-root system such that Hypotheses~1 and 2
hold. Then $G_{\tilde\Delta} = \{\,g\in G\mid g(\widetilde\Delta) =
\widetilde\Delta\,\}$ is a complement to $P$ in its normaliser in $G$.
\end{thm}

\begin{proof}
If $g\in N_G(P)$ and $a\in \Sigma_P$, then $gr_ag^{-1} = r_{g(a)}$ and
$g(a)\in\Sigma_P = \widetilde\Phi^+\cup \widetilde\Phi^-$. Define $\ell(g) =
|N(g)|$ and set $H_0 = \{\,g \in N_G(P) \mid \ell(g) = 0\,\}$. For $g\in P$,
$\ell(g)$ coincides with the usual length function for Coxeter groups and
therefore $P\cap H_0 = 1$.

For each coset of $P$ in $N_G(P)$ choose an element $g$ in the coset such
that $\ell(g)$ is minimal.  If $g(a)\in\widetilde\Phi^-$ for some
$a\in\Delta$, then from Lemma \ref{lemma:length} $\ell(gr_a) = \ell(g) - 1$,
contradicting the choice of $g$.  Thus $g(a)\in\widetilde\Phi^+$ for all
$a\in\Delta$.

The set $\Phi^+$ is the disjoint union of pairwise orthogonal irreducible
components $\Phi_1^+$, $\Phi_2^+$, \dots,~$\Phi_k^+$ and $\widetilde\Phi^+$
is the disjoint union of the sets $\gamma^j\Phi_i^+$ for $1\le j\le h$ and
$1\le i\le k$. If $a,b\in \Delta$ and $(a,b)\ne 0$, then
$g(\Delta)\subseteq\widetilde\Phi^+$ implies $g(a), g(b)\in\gamma^j\Phi_i^+$
for some $i$ and $j$. A connectedness argument now shows that
$g(\Delta\cap\Phi_i^+)\subseteq\gamma^j\Phi_i^+$. Every element of
$a\in\Phi_i^+$ is a positive linear combination of the roots in
$\Delta\cap\Phi_i^+$ and so $g(\Delta)\subseteq\widetilde\Phi^+$ implies
$g(\widetilde\Phi^+) = \widetilde\Phi^+$. Therefore $g\in H_0$.

To prove that $H_0$ is a subgroup we show that $h_1^{-1}h_2\in H_0$ for all
$h_1,h_2\in H_0$.  By way of contradiction suppose that
$h_1^{-1}h_2(a)\in\widetilde\Phi^-$ for some $a\in\Phi^+$. Then
$h_1^{-1}h_2(a) = \lambda b$ for some $b\in \widetilde\Phi^+$, where
$-\lambda\in\(\gamma\)$. But then $h_2(a) = \lambda h_1(b)$ and $h_1(b),
h_2(a) \in \widetilde\Phi^+$. This contradiction completes the proof that
$H_0$ is a subgroup.

If $g\in G_{\tilde\Delta}$, then $g\in N_G(P)$ and $g(\widetilde\Delta) =
\widetilde\Delta$ whence $\ell(g) = 0$ and so $g\in H_0$.  Conversely, by
definition, if $g\in H_0$ and $a\in\Phi^+$, then $g(a)\in\widetilde\Phi^+$
and it follows that $g(\widetilde\Phi^+) = \widetilde\Phi^+$. Then
$\widetilde\Delta\cap\gamma^j\Phi_i^+$ is the set roots in $\gamma^j\Phi_i^+$
that cannot be written as a nontrivial positive combination of other roots in
$\gamma^j\Phi_i^+$.  Therefore $g(\widetilde\Delta) = \widetilde\Delta$ and
hence $H_0 =  G_{\tilde\Delta}$.
\end{proof}

\begin{cor}
Suppose that $P$ is the standard parabolic subgroup $P_{(0,\lambda)}$ in $G =
G(m,p,n)$ and that $m$ is odd or twice an odd integer. Then $N_G(P)$ is the
semidirect product of $P$ by the stabiliser in $G$ of a set of roots of~$P$.
\end{cor}

\begin{proof}
The condition on $m$ ensures that Hypothesis 2 is satisfied. The parabolic
subgroup $P$ is of Coxeter type and for the simple roots we may take vectors
$e_{i-1}-e_i$ for suitable values of~$i$.  Thus Hypothesis 1 is satisfied and
therefore the stabiliser of $\widetilde\Delta$ is a parabolic complement.
\end{proof}

Next consider the standard parabolic subgroup $P = P_{(n_0,\lambda)}$, where
$n_0 \ne 0$. We may write $P = P_0\times P_1$, where $P_0 = G(m,p,n_0)$ and
where $P_1$ is the parabolic subgroup $P_{(0,\lambda)}$ in $G_1 =
G(m,p,n-n_0)$. Assume Hypothesis 2 and let $\widetilde\Delta$ be the roots of
$P_1$ defined in \eqref{eqn:extend}. If $\lambda = (n_1,\dots,n_d)$ and $e =
\gcd(p,n_1,\dots,n_d)$, let $\theta = \zeta_m^{me/p}$ and set
\[
  \Delta_0 = \{\, \theta^j e_1 - e_2,\theta^j e_1 - \zeta_m e_2
   \mid 0\le j < p/e \,\} \cup \{\, e_{i-1}-e_i \mid 3\le i \le n_0 \,\}.
\]
If $p\ne m$, adjoin the roots $\{\, \theta^j e_1 \mid  0\le j < p/e \,\} \cup
\{\,e_i \mid 2\le i\le n_0\,\}$ to $\Delta_0$. The reflections whose roots belong
to $\Delta_0\cup\widetilde\Delta$ generate $P$.

\begin{thm}
Suppose that Hypothesis 2 holds and that $P$ is the standard
parabolic subgroup $P_{(n_0,\lambda)}$ in $G = G(m,p,n)$ where $n_0 \ne 0$,
$\lambda = (n_1,\dots,n_d)$ and $e = \gcd(p,n_1,\dots,n_d)$. If $m/p$ and
$p/e$ are coprime, $N_G(P)$ is the semidirect product of $P$ by the setwise
stabiliser of $\Delta_0\cup\widetilde\Delta$ in $G$.
\end{thm}

\begin{proof}
In the notation established above we have $N_{G_1}(P_1) = P_1H_1$, where
$H_1$ stabilises $\widetilde\Delta$ and fixes every root in $\Delta_0$.
It follows from Corollary \ref{cor:orders} that the index of $PH_1$ in
$N_G(P)$ is $p/e$. If $p = e$, then $H_1 = G_{\Delta_0\cup\widetilde\Delta}$
is a complement to $P$ in $N_G(P)$.

If $p\ne e$, let $H = G_{\Delta_0 \cup\tilde\Delta}$ and suppose that $m/p$
and $p/e$ are coprime. Then $H_1\subseteq H\subseteq N_G(P)$ and $H\cap P = 1$.
It follows from the proof of Theorem \ref{thm:main-imprim} that $N_G(P)$
contains an element $g$ such that $e_1$ is an eigenvector of $g$ with
eigenvalue $\theta = \zeta_m^{me/p}$ and $g(e_i) = e_i$ for $2\le i\le n_0$. The
construction of $g$ in the proof of Theorem \ref{thm:main-imprim} utilised
the elements $h^{(k)}(j)$, which act on the blocks $ \Gamma_j^k$ of
size $k$ as scalar matrices $\zeta I_k$.  However, if $m$ is even, the
construction needs to be modified so that $g$ fixes $\widetilde\Delta$.
This is achieved by replacing the scalar matrix $\zeta I_k$ by the matrix with
$\zeta$ along the anti-diagonal.  Then $g$ fixes $\Delta_0 \cup\tilde\Delta$
and since $m/p$ and $p/e$ are coprime, $H = \(H_1,g\)$ and $N_G(P) = PH$.
\end{proof}

\section{The primitive unitary reflection groups}\label{sec:primitive}
From now on suppose $G$ is a finite primitive unitary reflection group of
rank at least 2 acting on $V$.

\subsection{Parabolic subgroups of rank 1}
Throughout this section $P$ denotes a rank 1 parabolic subgroup of $G$. Thus
$P = \(r\)$, where $r$ is a reflection of order $k$, where $k$ is 2, 3, 4 or
5 and $r$ is not the square of an element of order 4. Reflections of orders 4
or 5 occur only when the rank of $G$ is~2.

\begin{lemma}
If $P$ is a parabolic subgroup of $G$ of rank 1, then $N_G(P) = C_G(P)$.
\end{lemma}

\begin{proof}
If $g\in N_G(P)$ and $P = \(r\)$, then $g^{-1}rg = r^\ell$ for some $\ell$.
If $a$ is a root of $r$, then $ra = \lambda a$ for some root of unity
$\lambda\ne 1$. Then $rg a = \lambda^\ell ga$ and it follows that $\ell = 1$.
Thus $g\in C_G(r) = C_G(P)$.  The converse is obvious.
\end{proof}

Our first approach to finding a complement to $P = \(r\)$ in its normaliser
is to characterise those rank 1 parabolic subgroups whose normaliser is the
semidirect product of the parabolic by the stabiliser of a set of roots.

To this end, let $a$ be a root of the reflection $r$, let $k$ be the order of
$r$ and set
\begin{equation}\label{eqn:rou}
  \rou = \{\, \xi\in\C \mid\text{$ga = \xi a$ for some $g\in N_G(P)$}\,\}
  \quad\text{and}\quad
  B = \{\,\xi^k a\mid \xi\in\rou\,\}.
\end{equation}
Then $\rou$ consists of roots of unity and it is a cyclic group whose order
$m$ is a multiple of $k$.

\begin{lemma}\label{lemma:stab}
If $m/k$ and $k$ are coprime, the stabiliser in $G$ of the set $B$
is a complement to $P$ in $N_G(P)$. Conversely, if the stabiliser of a set of
scalar multiples of $a$ is a complement to $P$ in $N_G(P)$, then $m/k$ and
$k$ are coprime.
\end{lemma}

\begin{proof}
We have $|B| = m/k$ and the set $\{\,\xi a\mid \xi\in\rou\,\}$ is the
disjoint union of the images of $B$ under the action of $N = N_G(P)$. Thus
$H = G_B$ is contained in $N$ and its index in $N$ is $k$. The assumption
that $m/k$ and $k$ are coprime implies $H \cap P = 1$ and since $k = |P\,|$
it follows that $N = P\sdprod H$.

Conversely, suppose that $H = G_D$ is a complement to $P$ in $N$, where $D$
is a set of scalar multiples of $a$. Without loss of generality we may
suppose that $a\in D$ and that $D$ is an orbit of $H$.  Then $D = \rou_0 a$,
where
\[
  \rou_0 = \{\,\lambda\in\C\mid ha = \lambda a\text{ for some }h\in H\,\}
\]
is a subgroup of $\rou$, hence cyclic. Let $\theta$ be a generator of
$\rou_0$ and suppose that $ra = \zeta a$.  The elements of $\rou$ are the
products $\zeta^i\theta^j$. If $\zeta^i\theta^j = 1$, then $r^i$ preserves
$D$ and hence $r^i = 1$ since $P\cap H = 1$.  Thus $m = k|D|$.  If $k$ and
$m/k$ are not coprime, $\(\zeta\)$ and $\(\theta\)$ have non-trivial
intersection, contrary to what we have just proved.  This completes the
proof.
\end{proof}

The values of $k$ and $m$ for all primitive unitary reflection groups can be
obtained from Tables D.1 and D.2  of \cite{lehrer-taylor:2009} and the
descriptions of the rank 2 reflection groups in
\cite[Chapter~6]{lehrer-taylor:2009}.

The lemma just proved provides a parabolic complement for 22
of the 36 conjugacy classes of parabolic subgroups in the primitive
reflection groups of rank 2 and for the parabolic subgroups of rank 1 in all
but two of the primitive groups of rank at least~3. The exceptions (for rank
at least 3) are the Shephard and Todd groups $G_{29}$ and $G_{31}$, each of
which has a single class of reflections. In both cases the reflections are
involutions and the order of $\rou$ is~4.

\begin{example}
The Shephard and Todd group $G = G_{25}$ is the semidirect product of an
extraspecial group of order 27 and exponent 3 by the group $\SL(2,3)$ (see
\cite[Theorem 8.42]{lehrer-taylor:2009}). A rank 1 parabolic subgroup $P$ is
generated by a reflection $r$ of order 3 and $P$ has 12 conjugates, whence
$|N_G(P)| = 54$. Let $a$ denote a root of $r$. The order of the group $\rou$
of \eqref{eqn:rou} is 6 and therefore, by Lemma \ref{lemma:stab}, the
stabiliser of $\{a,-a\}$ is a parabolic complement. Since
$N_G(P) = C_G(P)$ it can be seen that $N_G(P)$ has an elementary abelian
group of order 27 and three complements to $P$ in $N_G(P)$.  The complements
not equal to $P$ act transitively on the 6 roots $(-\omega)^j a$ ($0\le j<
6$).  Thus $H$ is the unique complement which is the stabiliser of a set of
roots.  The elements of order 6 in $H$ fix no non-zero vectors and therefore
they are not 1-regular (in the sense of Springer \cite{springer:1974}).  In
particular, the fact that the elements of a parabolic complement in a
Coxeter group are 1-regular does not generalise to the unitary case.
\end{example}

\medskip
In order to deal with the groups not covered by Lemma \ref{lemma:stab} we
pursue other approaches to finding a complement.

\begin{lemma}\label{lemma:det}
Suppose that $P = \(r\)$, where $r$ is a reflection of order $k$ and let $d$
be the order of the image of $\det : N_G(P)\to \C^\times$. Then $N_G(P) =
P\times H$ where $H = \{\,g\in N_G(P)\mid \det(g)^{d/k} = 1\,\}$ except when
$P$ is a parabolic subgroup of order 2 in $G_9$ or $G_{11}$.
\end{lemma}

\begin{proof}
The image of $\det : N_G(P) \to \C$ is cyclic and its order is divisible by
$k$. It follows immediately that the index of $H = \{\,g\in N_G(P)\mid
\det(g)^{d/k} = 1\,\}$ in $N_G(P)$ is $k$ and $P\cap H = 1$ if and only if
$k$ and $d/k$ are coprime. To complete the proof observe that $d$ divides
the least common multiple of the orders of the reflections in $G$ and
therefore, from the tables in \cite{lehrer-taylor:2009}, $d/k$ and $k$ are
coprime except when $k = 2$ and 4 divides $d$.  Among the primitive
reflection groups only $G_9$ and $G_{11}$ contain parabolic subgroups of
orders 2 and~4.
\end{proof}

The exceptions in $G_9$ and $G_{11}$ to Lemma \ref{lemma:det} are also
exceptions to Lemma \ref{lemma:stab}.  We deal with these exceptions by
showing that in general, for almost all maximal parabolic subgroups $P$ in a
primitive reflection group $G$, the centre of $G$ is a parabolic complement
for~$P$.

\begin{thm}\label{thm:centre}
If $P$ is a maximal parabolic subgroup of the primitive reflection group $G$,
then $N_G(P) = P\times Z(G)$ except in the following cases. (See
\S\ref{sec:tab} for the definition of the \emph{type} of a parabolic
subgroup.)
\begin{enumerate}[(i)]
\item $G$ is $G_{13}$ or $G_{15}$ and $P$ is a parabolic subgroup
of order 2 with 6 conjugates in $G$.
\item $G = G_{25}$ and the type of $P$ is $2L_1$.
\item $G = G_{33}$ and the type of $P$ is $D_4$.
\item $G = G_{35}$ and the type of $P$ is $A_5$ or $A_1+2A_2$.
\end{enumerate}
\end{thm}

\begin{proof}
The index of a parabolic subgroup in its normaliser is tabulated in
\cite{orlik-solomon:1982,orlik-solomon:1983} and the order of $Z(G)$ can be
obtained from \cite[Table D.3]{lehrer-taylor:2009} using \cite[Corollary
3.24]{lehrer-taylor:2009}.  Since $P\cap Z(G) = 1$, it can be seen that the
only exceptions to $N_G(P) = P\times Z(G)$ are those listed.
\end{proof}

\begin{rem}\leavevmode
\begin{enumerate}[\enspace(i)]
\item Parabolic complements for the the parabolic subgroups of types $D_4$ in 
    $G_{33}$ and $A_5$ and $A_1+2A_2$ in $G_{35}$ follow from Theorem 
    \ref{thm:cox} below.
\item Suppose that $P$ is a parabolic subgroup of type $2L_1$ in $G = G_{25}$
    and let $Q$ be the pointwise stabiliser of $\Fix_V(P)^\perp$. From
    Table \ref{tbl:rk3} of \S\ref{sec:tab} the type of $Q$ is $L_1$,
    $|N| = 54$ and there is an involution $t\in N$, which commutes with $Q$
    and interchanges the root lines of $P$. Then $H = Q\(t\)$ is a
    complement to $P$ in $N$.
\end{enumerate}
\end{rem}
    

\begin{thm}\label{thm:rk1}
Suppose that $P$ is a parabolic subgroup of rank 1 in a primitive unitary
reflection group $G$ of rank at least 2 and let $Q$ be the pointwise
stabiliser of $\Fix_V(P)^\perp$.
\begin{enumerate}[(i)]
\item There exists a complement $H$ to $P$ in its normaliser.
\item Let $k = |P\,|$ and let $m$ be the order of the group $\rou$ defined
    in \eqref{eqn:rou}.  If $k$ is coprime to $m/k$, $H$ can be chosen to
    be the stabiliser of a set of roots and consequently $Q\subseteq H$.
    In particular this holds if the rank of $G$ is at least 3 and $G$ is
    neither $G_{29}$ nor~$G_{31}$.
\item Except when $G$ is $G_{27}$, $G_{29}$, $G_{33}$ or $G_{34}$, the
    group $H$ acts on $\Fix_V(P)$ as a unitary reflection group.  If the
    rank of $G$ is 2, then $H$ is cyclic. If the rank is at least 3, the
    isomorphism types of $Q$, $H^\circ$ and $H$ are given in the tables of
    \S\ref{sec:tab}.
\end{enumerate}
\end{thm}

\begin{proof}
(i)\enspace The existence of $H$ is a consequence of Lemma \ref{lemma:det}
and Theorem \ref{thm:centre}.

\smallskip
\noindent(ii)\enspace Suppose that the rank of $G$ is at least 3. Then the
order of a reflection in $G$ is either 2 or 3. From
\cite[Table~D.2]{lehrer-taylor:2009} the order of the group of roots of unity
of the ring of definition of $G$ is:
\begin{enumerate}[\quad(a)]
\item 2 for $G_{23}$, $G_{24}$, $G_{28}$, $G_{30}$, $G_{35}$, $G_{36}$ and
    $G_{37}$;
\item 6 for $G_{25}$, $G_{26}$, $G_{27}$, $G_{32}$, $G_{33}$ and $G_{34}$;
\item 4 for $G_{29}$ and $G_{31}$.
\end{enumerate}

Thus the order of the group $\rou$ defined in \eqref{eqn:rou} is 2 or 6,
except for reflections in $G_{29}$ and $G_{31}$ where it is~4. Therefore,
from Lemma \ref{lemma:stab}, except for $G_{29}$ and $G_{31}$, a complement
to $P$ may be chosen to contain $Q$.

\smallskip
\noindent(iii)\enspace If the rank of $G$ is 2, then $H$ is cyclic and
thus it acts as a reflection group on $\Fix_V(P)$. Furthermore, in all
cases except when the order of $P$ is 3 and $G$ is $G_4$ or $G_7$, the
groups $P$ and $Q$ have the same order (see \cite[Chapter 6]{lehrer-taylor:2009}).
In the two exceptional cases $Q = 1$.

If the rank of $G$ is at least 3, full details of the structure of $H$ can
be found in \cite{muralee:2005} and verified using the \Magma\ \cite{magma:1997}
code located at \url{http://www.maths.usyd.edu.au/u/don/details.html#programs}.
\end{proof}

From now on we suppose that the rank of $G$ is at least 3 and the rank of
$P$ is at least~2.

\subsection{Parabolic subgroups of Coxeter type}
From the tables in \cite{taylor:2012} or
\cite{orlik-solomon:1982,orlik-solomon:1983} there are 159 conjugacy classes
of parabolic subgroups of rank at least 2 in the 15 primitive reflection
groups of rank at least 3. All except 19 of these are of Coxeter type. We
shall see that, except for some instances when the ring of definition is the
Gaussian integers $\Z[i]$, the existence of a parabolic complement for a
parabolic subgroup of Coxeter type is a consequence of Theorem~\ref{thm:paracomp}.

The primitive unitary reflection groups which contain reflections of order 3
do not have parabolic subgroups of Coxeter type and rank at least~2.
Therefore, throughout this section we may suppose that $G$ is generated by
reflections of order~2.

From \cite[Theorem 8.30]{lehrer-taylor:2009}, except for $G_{28}$---the
Coxeter group of type $F_4$---there is an $A$-root system $\Sigma$ for $G$
such that $(a,a) = 2$ for all $a\in\Sigma$, where $A$ is the ring of definition
of $G$.  For the groups $G$ under consideration $\rou$ is the group of $m$th
roots of unity, where $m$ is 2, 4 or~6.

The set $\mathcal L$ of 1-dimensional subspaces spanned by the vectors in
$\Sigma$ is the \emph{line system} of $G$. Let $P$ be a parabolic subgroup of
$G$ of Coxeter type, let $\mathcal M$ be the subset of $\mathcal L$ spanned
by the roots of reflections in $P$ and let $\Sigma_P$ be the roots in
$\Sigma$ which belong to the lines in $\mathcal M$. Let $\Phi$ be the
standard root system of $P$ as defined, for example, in \cite{bourbaki:1968}.
The roots of $P$ are eigenvectors of its reflections and so we may assume
that the lines in $\mathcal M$ are spanned by roots in $\Phi$. Our first task
is to determine when Hypothesis 1 is satisfied.

\begin{lemma}\label{lemma:no4}
If $P$ does not contain reflections $r$ and $s$ such that the
order of $rs$ is 4, the root system $\Phi$ may be realised as a subset of
$\Sigma_P$.
\end{lemma}

\begin{proof}
It is a consequence of our assumptions that for $a,b\in \Sigma_P$ the order
of $r_ar_b$ is 1, 2, 3 or 5.  From \cite[Lemma 7.7]{lehrer-taylor:2009} if
$|r_ar_b| = 2$, then $(a,b) = 0$;  if $|r_ar_b| = 3$, then $(a,b)\in\rou$
whereas if $|r_ar_b| = 5$, then $(a,b) = \alpha\tau$ or $(a,b) =
\alpha\tau^{-1}$ for some $\alpha\in\rou$ and where $\tau^2 = \tau+1$.

Let $\mathcal M = \{\ell_1,\ell_2,\dots,\ell_t\}$ and for $1\le j\le t$
choose $a_j\in \ell_j\cap \Sigma_P$. Taking each root $a_i$ in turn, if
$(a_i,a_j) \ne (a_j,a_i)$ replace $a_j$ by $\alpha a_j$, where $\alpha =
(a_i,a_j)/|(a_i,a_j)|$.  It follows from \cite[\S1.6]{popov:1982} that this
process results in roots $a_i\in \ell_j\cap \Sigma_P$ such that all inner
products $(a_i,a_j)$ are real. These roots with their negatives form the
standard root system $\Phi$ of $P$.
\end{proof}

\begin{lemma}
If $P$ contains reflections $r$ and $s$ such that the order of $rs$ is 4, the
standard root system $\Phi$ of $P$ is the disjoint union of its long roots
$\Phi_\ell$ and its short roots $\Phi_s$.  If $G = G_{28}$, then $\Phi =
\Sigma_P$. In all other cases $\Phi_\ell$ may be realised as a subset of
$\Sigma_P$ and there exists $\gamma$ in $A$, the ring of definition, such
that $\gamma\bar\gamma = 2$ and $\Phi_s$ may be realised as a  subset
of~$(\gamma/2)\Sigma_P$.
\end{lemma}

\begin{proof}
We may suppose that $G \ne G_{28}$. It follows from \cite[Theorem
8.30]{lehrer-taylor:2009} that $A$ is either $\Z[\lambda]$, $\Z[\omega,\tau]$
or $\Z[i]$, where $\lambda^2 + \lambda + 2 = 0$, $\omega^2+\omega+1 = 0$,
$\tau^2 - \tau - 1 = 0$ and $i^2 + 1 = 0$. We may take $\gamma$ to be
$\lambda$, $\omega+\tau$ or~$i+1$, respectively.

For all $a\in\Sigma_P$ such that $a$ spans a short root line in $\mathcal M$,
replace $a$ by $(\gamma/2)a$.  We complete the proof by applying the
construction of the previous lemma.
\end{proof}

\begin{rem}
Of the 15 primitive reflection groups of rank at least 3 only $G_{24}$,
$G_{27}$, $G_{28}$ and $G_{29}$ contain parabolic subgroups of Coxeter type
that satisfy the hypotheses of this lemma and only types $B_2$ and~$B_3$
occur.
\end{rem}

From now on suppose that $\Phi$ is realised as in the previous two lemmas and
let $\Delta$, $\widetilde\Delta$ and $\widetilde\Phi^+$ be defined as in
\eqref{eqn:extend}.

\begin{thm}\label{thm:cox}
Suppose that $G$ is a primitive reflection group of rank at least 3 and that
the group of units of its ring of definition has order 2 or 6. If $P$ is a
parabolic subgroup of Coxeter type in $G$, then $G_{\tilde\Delta}$ is a
complement to $P$ in~$N_G(P)$.
\end{thm}

\begin{proof}
It is clear that Hypothesis 2 holds. If there are no reflections $r,s$ in $P$
such that the order of $rs$ is 4, Hypothesis 1 is a consequence of
Lemma~\ref{lemma:no4} and the result follows from Theorem \ref{thm:paracomp}.
This is also the case for $G_{28}$, the Coxeter group of type $F_4$.

Thus we are reduced to showing that the result holds when $P$ has type $B_2$
and $G$ is either $G_{24}$ or $G_{27}$.  In this case $\Delta = \{a,b\}$
where $a$ is a long root and $b$ is a short root.  As in the proof of Theorem
\ref{thm:paracomp}, $\{\,g\in G\mid g(\widetilde\Delta) =
\widetilde\Delta\,\}$ is a complement to $P$ in~$N_G(P)$.
\end{proof}

\begin{thm}\label{thm:gauss}
If $\rou = \rou_4$, then $G$ is $G_{29}$ or $G_{31}$ and if $P$ is a
parabolic subgroup of Coxeter type in $G$, we have $N_G(P) = P\sdprod H$ for
some subgroup $H$. The complement $H$ can be chosen to be the setwise
stabiliser of a set of roots of reflections which generate $P$ if and only if
the type of $P$ is $A_2$, $B_2$ or $A_3$ (two classes) in $G_{29}$ or $A_2$
or $A_3$ in~$G_{31}$.
\end{thm}

\begin{proof}
Let $N = N_G(P)$ and suppose at first that $P = \(r\)\times P_1$ where $r$ is
a reflection and $P_1$ is generated by reflections. We shall show that the
existence of a set $J$ of roots of $P$ whose setwise stabiliser in $N$ is a
complement to $P$ in $N$ leads to a contradiction. If $J$ does not contain a
root of $r$, then $r$ fixes $J$ and hence $N_J\cap P \ne 1$. If $J$ contains
a root $a$ of $r$, then $N$ contains an element $g$ such that $g(a) = ia$ and
if $N = N_JP$ we may write $g = hg_1g_2$, where $h\in N_J$, $g_1\in P_1$ and
$g_2\in\(r\)$. It follows that $h(a) = \pm ia$ and hence $-a\in J$. But then
$r(J) = J$ and again we reach the contradiction $N_J\cap P\ne 1$.  This
proves that if $P$ is a parabolic subgroup of type $2A_1$ or $A_1+A_2$, a
complement to $P$ in its normaliser cannot be the stabiliser of a set of
roots of~$P$.

If $G$ is $G_{29}$ and the type of $P$ is $B_3$, the pointwise stabiliser $Q$
of $\Fix_V(P)$ has order 2 and there are no elements $g\in N$ of order 4 such
that $g^2\in Q$. It follows that there is no set of roots of $P$ whose
setwise stabiliser is a complement to $P$ in $N$.

For all other parabolic subgroups of Coxeter type we shall show that a
complement can be obtained as the stabiliser of a set of roots.

If the type of $P$ is $A_2$ and if $a_1$ and $a_2$ are simple roots we set $J
= \{\pm a_1,\pm ia_2\}$. If the type of $P$ is $A_3$ and if $a_1$, $a_2$ and
$a_3$ are simple roots with Cartan matrix
$\left(\begin{smallmatrix}\pminus2&-1&\pminus0\\-1&\pminus2&-1\\
\pminus0&-1&\pminus2\end{smallmatrix}\right)$, we set $J = \{\pm a_1,\pm
a_2,\pm ia_2, \pm ia_3\}$.

Finally, if $G$ is $G_{29}$ and the type of $P$ is $B_2$, we may choose
simple roots $a_1$ and $a_2$ for $P$ with Cartan matrix
$\left(\begin{smallmatrix}\pminus2&-2\\-1&\pminus2\end{smallmatrix}\right)$
such that $a_1$, $(-i+1)a_2$ belong to the root system $\Sigma_P$ and then
set
\[
  J = \{a_1,(-i+1)a_2,i(a_1+2a_2),-(i+1)(a_1+a_2)\}.
\]
In all cases a direct calculation shows that $N = P\sdprod N_J$.
\end{proof}

\subsection{Parabolic subgroups of non-Coxeter type}
There are 20 conjugacy classes of non-Coxeter parabolic subgroups $P$ of rank
at least 2 in the primitive reflection groups of rank at least 3.  Thirteen
of these are maximal and we have shown in Theorem \ref{thm:centre} that
except for the parabolic subgroups of type $2L_1$ in $G_{25}$, the centre of
$G$ is a parabolic complement for $P$. However, it turns out that in
almost all cases there is a complement to $P$ which is the stabiliser of a
set of roots and which therefore contains the pointwise stabiliser $Q$ of the
orthogonal complement of the space of fixed points of~$P$.

\medskip
In preparation for the proofs that follow we represent the roots of
reflections as row vectors and the elements of $G$ as matrices acting on the
right.  This is consistent with the conventions of the \Magma\
\cite{magma:1997} code for unitary reflection groups.  If $G$ is generated by
reflections $r_1$, $r_2$, \dots,~$r_\ell$ with roots $a_1$, $a_2$,
\dots,~$a_\ell$, there are \emph{coroots} $b_1$, $b_2$, \dots,~$b_\ell$ such
that for all row vectors $v$ we have
\[
  v r_j = v - vb_j^\top a_j.
\]
The matrices $A$ and $B$ with rows $a_1$, $a_2$, \dots,~$a_\ell$ and $b_1$,
$b_2$, \dots,~$b_\ell$ are basic root and coroot matrices for $G$.  The
complex Cartan matrix of $A$ and $B$ is $AB^\top$.  We may suppose that the
roots $a_1$, $a_2$, \dots,~$a_n$, where $n$ is the rank of $G$, are the
standard basis vectors. If $\ell = n$, then $C = B^\top$ and the construction
of \cite[1.35]{lehrer-taylor:2009} produces the reflections $r_1$, $r_2$,
\dots,~$r_n$.  In this case it follows from \cite{taylor:2012} that every
subset of $\{r_1,r_2,\dots,r_n\}$ generates a parabolic subgroup.

For $G_{31}$ we have $n = 4$ and $\ell = 5$ and the root and coroot matrices
are given in Table \ref{tbl:gausscart}. The root system may be obtained by
taking the images of the basic roots under the action of $G$.  The Cartan
matrices in Table~\ref{tbl:eiscart} and the roots and coroots in Table
\ref{tbl:gausscart} have been chosen so that the reflections satisfy the
defining relations corresponding to the diagrams in
\cite{broue-malle-rouquier:1998}.

\begin{lemma}
If $P$ is a maximal parabolic subgroup of a primitive reflection group $G$ of
rank at least 3 and if $P$ is not of type $A_1+A_2$ or $B_3$ in $G_{29}$
nor of type $A_1+A_2$ or $G(4,2,3)$ in $G_{31}$, there is a set $J$ of roots
of $P$ such $N_G(P)$ is the semidirect product of $P$ by the setwise stabiliser
of $J$ in~$G$.
\end{lemma}

\begin{proof}
From Theorems \ref{thm:cox} and \ref{thm:gauss} we may suppose that $P$ is
not of Coxeter type.  If the ring of definition of $G$ is the Eisenstein
integers $\Z[\omega]$, a suitable set $J$  of roots is given in
Table~\ref{tbl:eis}.  Direct computation or the \Magma\ code referred
to in \S\ref{sec:intro} can be used to check that the stabiliser of $J$ is a
parabolic complement for~$P$.

This leaves the groups $G_{29}$ and $G_{31}$ to be considered.  If $P$ is a
parabolic subgroup of type $G(4,4,3)$ in $G_{29}$ we may suppose that $P$ is
generated by the reflections $r_1$, $r_2$ and $r_3$ obtained from the root
and coroot matrices of Table~\ref{tbl:gausscart}.  The stabiliser of $J = \{
\pm(0,0,i,0), \pm(i,0,1,0), \pm(0,i,i-1,0), \pm(i,1,i+1,0)\}$ is a parabolic
complement.

Finally suppose that $P$ is a parabolic subgroup of type $G(4,2,3)$ in
$G_{31}$. Then $N_G(P) = P\times Z(G)$ and $Z(G)$ is cyclic of order 4.
Furthermore, $\Fix_V(P)$ is spanned by a root and the pointwise stabiliser
$Q$ of $\Fix_V(P)^\perp$ has order 2.  It follows from Lemma~\ref{lemma:det}
that $C_G(Q) = Q \times K$ for some $K$. If $\(g\)$ is a complement to $P$ in
$N_G(P)$ and $Q\subset \(g\)$, then $g\in C_G(Q)$ and this contradiction
completes the proof.
\end{proof}

\begin{thm}\label{thm:prim3}
Suppose that $P$ is a parabolic subgroup of a primitive reflection group $G$
of rank at least~3. Then the normaliser $N$ of $P$ is the semidirect product of
$P$ and a subgroup $H$. Furthermore, there is a set $J$ of roots of $P$ such
that $H$ is the setwise stabiliser of $J$ in $G$ except when
\begin{enumerate}[(i)]
\item $G = G_{29}$ and the type of $P$ is $A_1$, $2A_1$, $A_1+A_2$ or
    $B_3$;
\item $G = G_{31}$ and the type of $P$ is $A_1$, $2A_1$, $A_1+A_2$,
    $G(4,2,2)$ or $G(4,2,3)$.
\end{enumerate}
\end{thm}

\begin{proof}
From results proved so far we may assume that $P$ is neither rank 1, maximal,
nor of Coxeter type.  This leaves seven types to be considered. Sets of roots
$J$ for the parabolic subgroups of types $2L_1$ and $L_2$ in $G_{32}$, type
$G(3,3,3)$ in $G_{33}$ and types $G(3,3,3)$, $G(3,3,4)$ and $A_1+G(3,3,3)$ in
$G_{34}$ are given in Table~\ref{tbl:eis}. It can be checked directly that
the stabiliser of $J$ is a parabolic complement for~$P$.

The remaining case is a parabolic subgroup $P$ of type $G(4,2,2)$ in
$G_{31}$. Then $Q$ is also a parabolic subgroup of type $G(4,2,2)$.
Reflections $r_1$, $r_2$, \dots, $r_5$ which generate $G_{31}$ can be
obtained from the root and coroot matrices in Table~\ref{tbl:gausscart} and
we may suppose that $P = \(r_1,r_3,r_5\)$. The subgroup $H$ generated by
\[
  \begin{pmatrix}
    -1&0&0&0\\
     0&-i&1&1\\
     i&0&-i&0\\
     -1&0&i + 1&1\end{pmatrix}\quad\text{and}\quad
  \begin{pmatrix}
     -i&0&-1&0\\
      i + 1&0&-i&-i\\
      0&0&-1&0\\
      i - 1&-1&-i + 2&-i + 1\end{pmatrix}
\]
is a complement to both $P$ and $Q$ in $N = N_G(P)$. Thus $H$ acts faithfully
on $\Fix_V(P)$ and $\Fix_V(P)^\perp$ and its restrictions to these subspaces
are isomorphic to the Shephard and Todd group $G_8$, of order 96.

Let $S$ be a Sylow 3-subgroup of $N$.  Then $S$ normalises $Q$ and the index
of $M = PQS$ in $N$ is~2.  A direct calculation shows that the orders of the
elements in $N\setminus M$ are 4 and 8 and for $x\in N\setminus M$ the order
of $QS\(x\)$ is either 384 or 1\,536.  Thus no complement to $P$ in $N$
contains $Q$, hence there is no complement that stabilises a set of roots of
$P$.
\end{proof}

\begin{rem}
Let $P$ be a parabolic subgroup of type $G(3,3,3)$ in $G = G_{33}$. Then
$N = N_G(P)$ acts faithfully on the support of $P$ as the reflection group
$G_{26}$ and it acts on the space of fixed points of $P$ as $G_4$.  There
are three conjugacy classes of complements to $P$ in $N$ only one of which
is represented by the setwise stabiliser $H$ of a set of roots (see 
Table~\ref{tbl:eis}) of generators of $P$ and every element of $H$ is 1-regular.

For a parabolic subgroup $P$ of type $G(3,3,4)$ in $G_{33}$ we have
$N = N_G(P) = P\times Z$, where the centre $Z$ of $G_{33}$ has order 2. 
In this case $Z$ is the setwise stabiliser of the roots $\{\pm a_1,\pm a_2,
\pm a_3, \pm a_4\}$ but it is not 1-regular.  There are two other 
conjugacy classes of complements to $P$ in $N$, one of size 18, the other
of size 27; all of these complements are 1-regular and stabilisers of
a set of roots of $P$.
\end{rem}

\section{The tables}\label{sec:tab}
The labels for the conjugacy classes of parabolic
subgroups of rank at least 3 use the notation introduced by Cohen
\cite{cohen:1976}. The captions on the tables use both the Cohen and the
Shephard and Todd naming schemes.

A parabolic subgroup which is the direct product of irreducible reflection
groups of types $T_1$, $T_2$,\dots,~$T_k$ will be labelled
$T_1+T_2+\cdots+T_k$ and if $T_i = T$ for all $i$ we denote the type by $kT$.

For the imprimitive reflection groups which occur in the tables we use the
Coxeter notation $B_n$ for the group $G(2,1,n)$, $D_n$ for $G(2,2,n)$ ($n \ge
4$) and $I_2(m)$ for $G(m,m,2)$ but otherwise use $G(m,p,n)$ as in
\S\ref{sec:imprimitive}.  For the Coxeter group of type $G_2$ we use $I_2(6)$
to avoid confusion with the Shephard and Todd notation $G_k$. The cyclic
reflection group of order $n$ is denoted by $\Z_n$ except that $A_1$ and
$L_1$ denote the groups of orders 2 and 3 respectively. If there are two
conjugacy classes of parabolic subgroups of type $T$ the classes are labelled
$T$ and $T'$.

If $A$ and $B$ are groups, $A\sdprod B$ denotes the semidirect product of $A$
by $B$, where $A$ is normalised by $B$.  The symbol $A\circ B$ denotes a
central product of $A$ and $B$.

The notation $H^\circ$ in the column headed $H$ indicates that the parabolic
complement $H$ coincides with the reflection group $H^\circ$.

Tables of conjugacy classes of the parabolic subgroups of the Coxeter groups
of types $E_6$, $E_7$, $E_8$, $F_4$, $H_3$ and $H_4$ can also be found in
\cite{douglass-etal:2011}.  The identification of the reflection subgroup
$H^\circ$ of the complement $H$ was verified using \Magma\ code
independently of the tables in \cite{amend-etal:2016} and
\cite{howlett:1980}.

\pagebreak

\begin{table}
\caption{Complex Cartan matrices of Eisenstein type}\label{tbl:eiscart}
\[
  G_{25} : \begin{pmatrix}
        1-\omega& -\omega^2&         0\\
        \omega^2&  1-\omega& \omega^2\\
               0& -\omega^2& 1-\omega
  \end{pmatrix}\quad
  G_{26} : \begin{pmatrix}
      1-\omega&   -\omega^2&  0\\
        \omega^2&  1-\omega& -1\\
             0& -1+\omega&  2
  \end{pmatrix}\quad
  G_{32} : \begin{pmatrix}
       1-\omega& \omega^2&         0&        0\\
      -\omega^2& 1-\omega& -\omega^2&        0\\
              0& \omega^2&  1-\omega& \omega^2\\
              0&        0& -\omega^2& 1-\omega
    \end{pmatrix}
\]

\[
   G_{33} : \begin{pmatrix}
        2&  -1&        0&       0&  0\\
       -1&   2&       -1& -\omega&  0\\
        0&  -1&        2&      -1&  0\\
        0& -\omega^2& -1&       2& -1\\
        0&   0&        0&      -1&  2
    \end{pmatrix}\quad
    G_{34} : \begin{pmatrix}
        2&        -1&  0&        0&  0&  0\\
       -1&         2& -1&  -\omega&  0&  0\\
        0&        -1&  2&       -1&  0&  0\\
        0& -\omega^2& -1&        2& -1&  0\\
        0&         0&  0&       -1&  2& -1\\
        0&         0&  0&        0& -1&  2
    \end{pmatrix}
\]
\end{table}

\begin{table}
\caption{Basic root and coroot matrices of Gaussian type}\label{tbl:gausscart}
\[
  A_{29} = \begin{pmatrix}
        1&0&0&0\\0&1&0&0\\0&0&1&0\\0&0&0&1\end{pmatrix}\quad
  B_{29} = \begin{pmatrix}
        2&  -1&  -i&  \pminus0\\
       -1&   2& i-1&  \pminus0\\
        i& -i-1&  \pminus2& -1\\
        0&   0&   -1& \pminus2
    \end{pmatrix}
\]

\[
  A_{31} = \begin{pmatrix}
        1&0&0&0\\0&1&0&0\\0&0&1&0\\0&0&0&1\\-1&0&i&0\end{pmatrix}\quad
  B_{31} = \begin{pmatrix}
             2 & -1 & -i+1 & 0 \\
             -1 & 2 & i & 0 \\
             i+1 & -i & 2 & -1 \\
             0 & 0 & -1 & 2 \\
             -i-1 & 0 & -i-1 & i
           \end{pmatrix}
\]
\end{table}

\begin{table}\small
\caption{Non-Coxeter parabolic complements in $G_n$ of Eisenstein type}\label{tbl:eis}
\begingroup
\medskip\centering
\begin{tabular}{lccl}
\toprule
  $n$&$P$&generators of $P$&\quad$J$\hskip2cm $N = P\sdprod N_J$ where $N = N_{G_n}(P)$\\
\midrule
25&$2L_1$&$r_1, r_3$&$\(-1\)(1,0,0)$, $\(-1\)(0,0,1)$\\
  &$L_2$&$r_1, r_2$&$\(\omega\)(1,0,0)$, $\(\omega\)(0,1,0)$\\
\midrule
26&$A_1+L_1$&$r_1, r_3$&$\(-1\)(1,0,0)$, $\(\omega\)(0,0,1)$\\
  &$L_2$&$r_1, r_2$&$\(\omega\)(1,0,0)$, $\(\omega\)(0,1,0)$\\
  &$G(3,1,2)$&$r_2, r_3$&$\(-1\)(0,1,0)$, $\(-1\)(0,0,1)$\\
\midrule
32&$2L_1$&$r_1, r_3$&$\(-1\)(1,0,0,0)$, $\(-1\)(0,0,1,0)$\\
  &$L_2$&$r_1, r_2$&$\(\omega\)(1,0,0,0)$, $\(\omega\)(0,1,0,0)$\\
  &$L_1+L_2$&$r_1, r_2, r_4$&$\(\omega\)(1,0,0,0)$, $\(\omega\)(0,-1,0,0),\(-1\)(0,0,0,1)$\\
  &$L_3$&$r_1, r_2, r_3$&$\(-1\)(1,0,0,0)$, $\(-1\)(0,\omega^2,0,0)$, $\(-1\)(0,0,\omega,0)$\\
\midrule
33&$G(3,3,3)$&$r_2, r_3, r_4$&$(0,1,0,0,0)$, $(0,0,1,0,0)$, $(0,0,0,\omega^2,0)$,\\
  &&&\quad$(0,1,1,-\omega^2,0)$, $(0,1,-\omega^2,\omega,0)$, $(0,-\omega,1,1,0)$,\\
  &&&\quad$(0,0,-\omega,-\omega,0)$, $(0,-\omega^2,0,-1,0)$\\
  &$G(3,3,4)$&$r_1, r_2, r_3, r_4$&$(1,0,0,0,0)$, $(0,1,0,0,0)$, $\(-1\)(0,0,1,0,0)$,\\
  &&&\quad$\(-1\)(0, 0, 0, 1, 0)$, $(-1, -1, 0, 0, 0)$,\\
\midrule
34&$G(3,3,3)$&$r_2, r_3, r_4$&$(0,1,0,0,0,0)$, $(0,0,1,0,0,0)$, $(0,0,0,\omega^2,0,0)$,\\
  &&&\quad$(0,1,1,-\omega^2,0,0)$, $(0,1,-\omega^2,\omega,0,0)$, $(0,-\omega,1,1,0,0)$,\\
  &&&\quad$(0,0,-\omega,-\omega,0,0)$, $(0,-\omega^2,0,-1,0,0)$\\
  &$G(3,3,4)$&$r_1, r_2, r_3, r_4$&$(0,1,0,0,0,0)$, $(-1,-1,0,0,0,0)$, $\(-1\)(0,0,0,1,0,0)$,\\
  &&&\quad$\(-1\)(0,0,\omega^2,0,0,0)$, $\(-1\)(0,0,\omega,\omega,0,0)$\\
  &$A_1+G(3,3,3)$&$r_2, r_3, r_4, r_6$&$(0,1,0,0,0,0)$, $(0,0,1,0,0,0)$, $(0,0,0,\omega^2,0,0)$,\\
  &&&\quad$(0,1,1,-\omega^2,0,0)$, $(0,1,-\omega^2,\omega,0,0)$, $(0,-\omega,1,1,0,0)$,\\
  &&&\quad$(0,0,-\omega,-\omega,0,0)$, $(0,-\omega^2,0,-1,0,0)$,\\
  &&&\quad$\(\omega\)(0,0,0,0,0,1)$\\
  &$A_1+G(3,3,4)$&$r_1, r_2, r_3, r_4, r_6$&$(0,1,0,0,0,0)$, $(-1,-1,0,0,0,0)$, $\(-1\)(0,0,0,1,0,0)$,\\
  &&&\quad$\(-1\)(0,0,\omega^2,0,0,0)$, $\(-1\)(0,0,\omega,\omega,0,0)$,\\
  &&&\quad$\(\omega\)(0,0,0,0,0,1)$\\
  &$G(3,3,5)$&$r_2, r_3, r_4, r_5, r_6$&$(0,\omega^2,0,0,0,0)$, $(0,0,1,0,0,0)$, $(0,0,0,0,1,0)$,\\
  &&&\quad$(0,-\omega,-\omega,0,0,0)$, $(0,-\omega,0,-\omega^2,0,0)$, $(0,1,1,-\omega^2,0,0)$,\\
  &&&\quad$(0,0,-\omega^2,-\omega^2,0,0)$, $(0,0,0,-\omega,-\omega,-\omega)$\\
  &$K_5$&$r_1, r_2, r_3, r_4, r_5$&$\(\omega\)(1,0,0,0,0,0)$, $\(\omega\)(0,1,0,0,0,0)$,
   $\(\omega\)(0,0,1,0,0,0)$,\\
  &&&\quad $\(\omega\)(0,0,0,1,0,0)$, $\(\omega\)(0,0,0,0,1,0)$\\
\bottomrule
\end{tabular}

\medskip
Notation: for a root $a$, interpret $\(-1\)a$ as $a$, $-a$ and interpret
$\(\omega\)a$ as $a$, $\omega a$, $\omega^2 a$, where $\omega^3 = 1$.\par
\endgroup

\smallskip

The reflections $r_1$, $r_2$, \dots, $r_d$ are the generators of $G_n$ of
rank $d$ corresponding to the complex Cartan matrices of
Table~\ref{tbl:eiscart} and the reflections with roots $J$ generate $P$. The
setwise stabiliser of $J$ in $G_n$ is a parabolic complement for~$P$.
\end{table}

\begin{table}\small
\caption{The parabolic subgroups in reflection groups $G_n$ of
rank 3}\label{tbl:rk3}
\medskip\centering
\begin{tabular}{lccrcc}
\toprule
  &$P$&$Q$&$|H^\circ|$&$H^\circ$&$H$\\
\midrule
$G_{23} = W(H_3)$&$A_1$&$2A_1$&4&$2A_1$&$H^\circ$\\
\addlinespace
&$2A_1$&$A_1$&2&$A_1$&$H^\circ$\\
&$A_2$&\textbf1&2&$A_1$&$H^\circ$\\
&$I_2(5)$&\textbf1&2&$A_1$&$H^\circ$\\
\midrule
$G_{24}=W(J_3^{(4)})$&$A_1$&$B_2$&8&$B_2$&$H^\circ$\\
\addlinespace
&$A_2$&\textbf1&2&$A_1$&$H^\circ$\\
&$B_2$&$A_1$&2&$A_1$&$H^\circ$\\
\midrule
$G_{25}=W(L_3)$&$L_1$&$2L_1$&18&$G(3,1,2)$&$H^\circ$\\
\addlinespace
&$2L_1$&$L_1$&6&$\Z_6$&$H^\circ$\\
&$L_2$&\textbf1&3&$L_1$&$H^\circ$\\
\midrule
$G_{26}=W(M_3)$&$A_1$&$L_2$&72&$G_5$&$H^\circ$\\
&$L_1$&$G(3,1,2)$&36&$G(6,2,2)$&$H^\circ$\\
\addlinespace
&$A_1+L_1$&\textbf1&6&$\Z_6$&$H^\circ$\\
&$L_2$&$A_1$&6&$\Z_6$&$H^\circ$\\
&$G(3,1,2)$&$L_1$&6&$\Z_6$&$H^\circ$\\
\midrule
$G_{27}=W(J_3^{(5)})$&$A_1$&$B_2$&8&$B_2$&\color{Red}$H^\circ\times\Z_3$\\
\addlinespace
&$A_2$&\textbf1&6&$\Z_6$&$H^\circ$\\
&$A'_2$&\textbf1&6&$\Z_6$&$H^\circ$\\
&$I_2(5)$&\textbf1&6&$\Z_6$&$H^\circ$\\
&$B_2$&$A_1$&6&$\Z_6$&$H^\circ$\\
\bottomrule
\end{tabular}
\end{table}

\begin{table}\small
\centering
\begin{minipage}[t]{6.8cm}
\caption{The parabolic subgroups of $G_{28} = W(F_4)$}\label{tbl:28}
\medskip\centering
\begin{tabular}{ccrcc}
\toprule
$P$&$Q$&$|H^\circ|$&$H^\circ$&$H$\\
\midrule
$A_1$&$B_3$&48&$B_3$&$H^\circ$\\
$A'_1$&$B'_3$&48&$B'_3$&$H^\circ$\\
\addlinespace
$2A_1$&$2A_1$&4&$2A_1$&$H^\circ$\\
$A_2$&$A'_2$&12&$I_2(6)$&$H^\circ$\\
$A'_2$&$A_2$&12&$I_2(6)$&$H^\circ$\\
$B_2$&$B_2$&8&$B_2$&$H^\circ$\\
\addlinespace
$A_1+A_2$&\textbf1&2&$A_1$&$H^\circ$\\
$(A_1+A_2)'$&\textbf1&2&$A_1$&$H^\circ$\\
$B_3$,&$A_1$&2&$A_1$&$H^\circ$\\
$B'_3$&$A'_1$&2&$A_1$&$H^\circ$\\
\bottomrule
\end{tabular}
\end{minipage}
\begin{minipage}[t]{8cm}
\caption{The parabolic subgroups of $G_{29} = W(N_4)$}\label{tbl:29}
\medskip\centering
\begin{tabular}{ccrcc}
\toprule
$P$&$Q$&$|H^\circ|$&$H^\circ$&$H$\\
\midrule
$A_1$&$B_3$&48&$B_3$&\color{Red}$H^\circ\circ\Z_4$\\
\addlinespace
$2A_1$&$2A_1$&16&$G(4,2,2)$&$H^\circ$\\
$A_2$&$A_1$&4&$2A_1$&\color{Red}$Q\times \Z_4$\\
$B_2$&$B_2$&32&$G(4,1,2)$&$H^\circ$\\
\addlinespace
$A_1+A_2$&\textbf1&4&$\Z_4$&$H^\circ$\\
$A_3$&\textbf1&4&$\Z_4$&$H^\circ$\\
$A'_3$&\textbf1&4&$\Z_4$&$H^\circ$\\
$G(4,4,3)$&\textbf1&4&$\Z_4$&$H^\circ$\\
$B_3$&$A_1$&4&$\Z_4$&$H^\circ$\\
\bottomrule
\end{tabular}
\end{minipage}
\end{table}

\begin{table}\small
\centering
\begin{minipage}[t]{7.7cm}
\caption{The parabolic subgroups of $G_{30} = W(H_4)$}\label{tbl:30}
\medskip\centering
\begin{tabular}{ccrcc}
\toprule
$P$&$Q$&$|H^\circ|$&$H^\circ$&$H$\\
\midrule
$A_1$&$H_3$&120&$H_3$&$H^\circ$\\
\addlinespace
$2A_1$&$2A_1$&8&$B_2$&$H^\circ$\\
$A_2$&$A_2$&12&$I_2(6)$&$H^\circ$\\
$I_2(5)$&$I_2(5)$&20&$I_2(10)$&$H^\circ$\\
\addlinespace
$A_1+A_2$&\textbf1&2&$A_1$&$H^\circ$\\
$A_1+I_2(5)$&\textbf1&2&$A_1$&$H^\circ$\\
$A_3$&\textbf1&2&$A_1$&$H^\circ$\\
$H_3$&$A_1$&2&$A_1$&$H^\circ$\\
\bottomrule
\end{tabular}
\end{minipage}
\begin{minipage}[t]{8.6cm}
\caption{The parabolic subgroups of $G_{31} = W(EN_4)$}\label{tbl:31}
\medskip\centering
\begin{tabular}{ccrcc}
\toprule
$P$&$Q$&$|H^\circ|$&$H^\circ$&$H$\\
\midrule
$A_1$&$G(4,2,3)$&384&$G(4,1,3)$&$H^\circ$\\
\addlinespace
$2A_1$&$2A_1$&32&$G(4,1,2)$&$H^\circ$\\
$A_2$&$A_2$&12&$I_2(6)$&\color{Red}$Q\times \Z_4$\\
$G(4,2,2)$&$G(4,2,2)$&96&$G_8$&$H^\circ$\\
\addlinespace
$A_1+A_2$&\textbf1&4&$\Z_4$&$H^\circ$\\
$A_3$&\textbf1&4&$\Z_4$&$H^\circ$\\
$G(4,2,3)$&$A_1$&4&$\Z_4$&$H^\circ$\\
\bottomrule
\end{tabular}
\end{minipage}
\end{table}

\begin{table}\small
\centering
\begin{minipage}[t]{7.7cm}
\caption{The parabolic subgroups of $G_{32} = W(L_4)$}\label{tbl:32}
\medskip\centering
\begin{tabular}{ccrcc}
\toprule
$P$&$Q$&$|H^\circ|$&$H^\circ$&$H$\\
\midrule
$L_1$&$L_3$&1296&$M_3$&$H^\circ$\\
\addlinespace
$2L_1$&$2L_1$&72&$G(6,1,2)$&$H^\circ$\\
$L_2$&$L_2$&72&$G_5$&$H^\circ$\\
\addlinespace
$L_1+L_2$&\textbf1&6&$\Z_6$&$H^\circ$\\
$L_3$&$L_1$&6&$\Z_6$&$H^\circ$\\
\bottomrule
\end{tabular}
\end{minipage}
\begin{minipage}[t]{8.1cm}
\caption{The parabolic subgroups of $G_{33} = W(K_5)$}\label{tbl:33}
\medskip\centering
\begin{tabular}{ccrcc}
\toprule
$P$&$Q$&$|H^\circ|$&$H^\circ$&$H$\\
\midrule
$A_1$&$D_4$&192&$D_4$&\color{Red}$H^\circ\sdprod\Z_3$\\
\addlinespace
$2A_1$&$3A_1$&48&$B_3$&$H^\circ$\\
$A_2$&$A_2$&18&$G(3,1,2)$&\color{Red}$H^\circ\times\Z_2$\\
\addlinespace
$A_1+A_2$&\textbf1&3&$L_1$&\color{Red}$H^\circ\times \Z_2$\\
$A_3$&$A_1$&4&$2A_1$&$H^\circ$\\
$3A_1$&$2A_1$&24&$G(6,3,2)$&$H^\circ$\\
$G(3,3,3)$&\textbf1&24&$L_2$&$H^\circ$\\
\addlinespace
$A_1+A_3$&\textbf1&2&$A_1$&$H^\circ$\\
$A_4$&\textbf1&2&$A_1$&$H^\circ$\\
$G(3,3,4)$&\textbf1&2&$A_1$&$H^\circ$\\
$D_4$&$A_1$&6&$\Z_6$&$H^\circ$\\
\bottomrule
\end{tabular}
\end{minipage}
\end{table}

\begin{table}\small
\begin{minipage}[t]{9cm}
\caption{The parabolic subgroups of $G_{34} = W(K_6)$}\label{tbl:34}
\medskip\centering
\begin{tabular}{ccrcc}
\toprule
$P$&$Q$&$|H^\circ|$&$H^\circ$&$H$\\
\midrule
$A_1$&$K_5$&\llap{51\,840}&$K_5$&\color{Red}$H^\circ\times\Z_3$\\
\addlinespace
$2A_1$&$D_4$&\llap{1\,152}&$F_4$&\color{Red}$H^\circ\times\Z_3$\\
$A_2$&$G(3,3,4)$&\llap{1\,944}&$G(3,1,4)$&\color{Red}$H^\circ\times \Z_2$\\
\addlinespace
$A_1+A_2$&$A_2$&54&\zwidth[45pt]{$G(3,1,2)+L_1$}&\color{Red}$H^\circ\times\Z_2$\\
$A_3$&$A_3$&48&$B_3$&\color{Red}$H^\circ\times\Z_3$\\
$3A_1$&$3A_1$&432&$G(6,3,3)$&$H^\circ$\\
$G(3,3,3)$&$G(3,3,3)$&\llap{1\,296}&$M_3$&$H^\circ$\\
\addlinespace
$A_1+A_3$&$A_1$&4&$2A_1$&\color{Red}$H^\circ\times\Z_3$\\
$G(3,3,4)$&$A_2$&36&$G(6,2,2)$&$H^\circ$\\
$A_4$&$A_1$&4&$2A_1$&\color{Red}$H^\circ\times\Z_3$\\
$2A_1+A_2$&\textbf1&36&$G(6,2,2)$&$H^\circ$\\
$2A_2$&\textbf1&36&$G(6,2,2)$&$H^\circ$\\
\zwidth{$A_1+G(3,3,3)$}&\textbf1&72&$G_5$&$H^\circ$\\
$D_4$&$2A_1$&72&$G(6,1,2)$&$H^\circ$\\
\addlinespace
$A_1+A_4$&\textbf1&6&$\Z_6$&$H^\circ$\\
$A_2+A_3$&\textbf1&6&$\Z_6$&$H^\circ$\\
\zwidth{$A_1+G(3,3,4)$}&\textbf1&6&$\Z_6$&$H^\circ$\\
$A_5$&\textbf1&6&$\Z_6$&$H^\circ$\\
$A'_5$&\textbf1&6&$\Z_6$&$H^\circ$\\
$D_5$&\textbf1&6&$\Z_6$&$H^\circ$\\
$G(3,3,5)$&\textbf1&6&$\Z_6$&$H^\circ$\\
$K_5$&$A_1$&6&$\Z_6$&$H^\circ$\\
\bottomrule
\end{tabular}
\end{minipage}
\begin{minipage}[t]{7cm}
\caption{The parabolic subgroups of $G_{35} = W(E_6)$}\label{tbl:35}
\medskip\centering
\begin{tabular}{ccrcc}
\toprule
$P$&$Q$&$|H^\circ|$&$H^\circ$&$H$\\
\midrule
$A_1$&$A_5$&\llap{720}&$A_5$&$H^\circ$\\
\addlinespace
$2A_1$&$A_3$&48&$B_3$&$H^\circ$\\
$A_2$&$2A_2$&36&$2A_2$&\color{Red}$A_2\wr\Z_2$\\
\addlinespace
$A_1+A_2$&$A_2$&6&$A_2$&$H^\circ$\\
$A_3$&$2A_1$&8&$B_2$&$H^\circ$\\
$3A_1$&$A_1$&12&$A_1+A_2$&$H^\circ$\\
\addlinespace
$A_1+A_3$&$A_1$&2&$A_1$&$H^\circ$\\
$A_4$&$A_1$&2&$A_1$&$H^\circ$\\
$2A_2$&$A_2$&12&$I_2(6)$&$H^\circ$\\
$2A_1+A_2$&\textbf1&2&$A_1$&$H^\circ$\\
$D_4$&\textbf1&6&$A_2$&$H^\circ$\\
\addlinespace
$A_1+A_4$&\textbf1&1&\textbf1&$H^\circ$\\
$D_5$&\textbf1&1&\textbf1&$H^\circ$\\
$A_1 + 2A_2$&\textbf1&2&$A_1$&$H^\circ$\\
$A_5$&$A_1$&2&$A_1$&$H^\circ$\\
\bottomrule
\end{tabular}
\end{minipage}
\end{table}

\begin{table}\small
\begin{minipage}[t]{8.05cm}
\caption{The parabolic subgroups of $G_{36} = W(E_7)$}\label{tbl:rk7}
\medskip\centering
\begin{tabular}{ccrcc}
\toprule
$P$&$Q$&$|H^\circ|$&$H^\circ$&$H$\\
\midrule
$A_1$&$D_6$&23\,040&$D_6$&$H^\circ$\\
\addlinespace
$A_2$&$A_5$&720&$A_5$&\color{Red}$H^\circ\times\Z_2$\\
$2A_1$&\zwidth{$A_1+D_4$}&768&$A_1+B_4$&$H^\circ$\\
\addlinespace
$A_3$&\zwidth{$A_1+A_3$}&96&$A_1+B_3$&$H^\circ$\\
$A_1+A_2$&$A_3$&24&$A_3$&\color{Red}$H^\circ\times\Z_2$\\
$3A_1$&$4A_1$&96&$A_1+B_3$&$H^\circ$\\
$3A'_1$&$D_4$&1\,152&$F_4$&$H^\circ$\\
\addlinespace
$A_4$&$A_2$&6&$A_2$&\color{Red}$H^\circ\times\Z_2$\\
$D_4$&$3A'_1$&48&$B_3$&$H^\circ$\\
$A_1+A_3$&$2A_1$&8&$3A_1$&$H^\circ$\\
$(A_1+A_3)'$&$A_3$&48&$B_3$&$H^\circ$\\
$2A_2$&$A_2$&24&$A_1+I_2(6)$&$H^\circ$\\
$2A_1+A_2$&$A_1$&8&$3A_1$&$H^\circ$\\
$4A_1$&$3A_1$&48&$B_3$&$H^\circ$\\
\addlinespace
$A_5$&$A_1$&4&$2A_1$&$H^\circ$\\
$A'_5$&$A_2$&12&$I_2(6)$&$H^\circ$\\
$D_5$&$A_1$&4&$2A_1$&$H^\circ$\\
$A_1+A_4$&\textbf1&1&\textbf1&\color{Red}$\Z_2$\\
$A_1+D_4$&$2A_1$&8&$B_2$&$H^\circ$\\
$A_2+A_3$&$A_1$&4&$2A_1$&$H^\circ$\\
$2A_1+A_3$&$A_1$&4&$2A_1$&$H^\circ$\\
$A_1+2A_2$&\textbf1&4&$2A_1$&$H^\circ$\\
$3A_1+A_2$&\textbf1&12&$I_2(6)$&$H^\circ$\\
\addlinespace
\zwidth{$A_1+A_2+A_3$}&\textbf1&2&$A_1$&$H^\circ$\\
$A_1+A_5$&\textbf1&2&$A_1$&$H^\circ$\\
$A_2+A_4$&\textbf1&2&$A_1$&$H^\circ$\\
$A_1+D_5$&\textbf1&2&$A_1$&$H^\circ$\\
$A_6$&\textbf1&2&$A_1$&$H^\circ$\\
$E_6$&\textbf1&2&$A_1$&$H^\circ$\\
$D_6$&$A_1$&2&$A_1$&$H^\circ$\\
\bottomrule
\end{tabular}
\end{minipage}
\begin{minipage}[t]{8.3cm}
\caption{The parabolic subgroups of $G_{37} = W(E_8)$}\label{tbl:rk8}
\medskip\centering
\begin{tabular}{ccrcc}
\toprule
$P$&$Q$&$|H^\circ|$&$H^\circ$&$H$\\
\midrule
$A_1$&$E_7$&\qquad\quad\llap{2\,903\,040}&$E_7$&$H^\circ$\\
\addlinespace
$2A_1$&$D_6$&46\,080&$B_6$&$H^\circ$\\
$A_2$&$E_6$&51\,840&$E_6$&\color{Red}$H^\circ\times\Z_2$\\
\addlinespace
$A_1+A_2$&$A_5$&720&$A_5$&\color{Red}$H^\circ\times\Z_2$\\
$3A_1$&\zwidth{$A_1+D_4$}&2\,304&$A_1+F_4$&$H^\circ$\\
$A_3$&$D_5$&3\,840&$B_5$&$H^\circ$\\
\addlinespace
$A_1+A_3$&\zwidth{$A_1+A_3$}&96&$A_1+B_3$&$H^\circ$\\
$2A_2$&$2A_2$&144&$2I_2(6)$&\color{Red}$I_2(6)\wr\Z_2$\\
$A_4$&$A_4$&120&$A_4$&\color{Red}$H^\circ\times\Z_2$\\
$2A_1+A_2$&$A_3$&96&$A_1+B_3$&$H^\circ$\\
$D_4$&$D_4$&1\,152&$F_4$&$H^\circ$\\
$4A_1$&$4A_1$&384&$B_4$&$H^\circ$\\
\addlinespace
$A_1+A_4$&$A_2$&6&$A_2$&\color{Red}$H^\circ\times\Z_2$\\
$A_2+A_3$&$2A_1$&16&$A_1+B_2$&$H^\circ$\\
$A_5$&\zwidth{$A_1+A_2$}&24&$A_1+I_2(6)$&$H^\circ$\\
$D_5$&$A_3$&48&$B_3$&$H^\circ$\\
$A_1+D_4$&$3A_1$&48&$B_3$&$H^\circ$\\
$2A_1+A_3$&$2A_1$&16&$A_1+B_2$&$H^\circ$\\
$A_1+2A_2$&$A_2$&24&$A_1+I_2(6)$&$H^\circ$\\
$3A_1+A_2$&$A_1$&24&$A_1+I_2(6)$&$H^\circ$\\
\addlinespace
$A_1+A_5$&$A_1$&4&$2A_1$&$H^\circ$\\
$A_2+A_4$&$A_1$&4&$2A_1$&$H^\circ$\\
$A_6$&$A_1$&4&$2A_1$&$H^\circ$\\
$2A_1+A_4$&\textbf1&4&$2A_1$&$H^\circ$\\
$E_6$&$A_2$&12&$I_2(6)$&$H^\circ$\\
$A_1+D_5$&$A_1$&4&$2A_1$&$H^\circ$\\
$2A_3$&\textbf1&8&$B_2$&$H^\circ$\\
$D_6$&$2A_1$&8&$B_2$&$H^\circ$\\
$A_2+D_4$&\textbf1&12&$I_2(6)$&$H^\circ$\\
$A_1+A_2+A_3$&$A_1$&4&$2A_1$&$H^\circ$\\
$2A_1+2A_2$&\textbf1&8&$B_2$&$H^\circ$\\
\addlinespace
$A_1+A_6$&\textbf1&2&$A_1$&$H^\circ$\\
$A_1+A_2+A_4$&\textbf1&2&$A_1$&$H^\circ$\\
$A_1+E_6$&\textbf1&2&$A_1$&$H^\circ$\\
$A_2+D_5$&\textbf1&2&$A_1$&$H^\circ$\\
$A_3+A_4$&\textbf1&2&$A_1$&$H^\circ$\\
$A_7$&\textbf1&2&$A_1$&$H^\circ$\\
$D_7$&\textbf1&2&$A_1$&$H^\circ$\\
$E_7$&$A_1$&2&$A_1$&$H^\circ$\\
\bottomrule
\end{tabular}
\end{minipage}
\end{table}

\clearpage
\section*{Acknowledgements}

This work is based on a portion of the first author's Ph.\,D.~thesis, which
was supervised by the second author at the University of Sydney. We thank 
an anonymous referee for several useful comments.

\bibliography{ParabolicComplements}
\bibliographystyle{abbrv}

\end{document}